\documentclass[a4paper]{amsart}
\usepackage[active]{srcltx}
\usepackage[all]{xy}
\usepackage{subfig}
\usepackage{hyperref}
\usepackage{mathrsfs}

\usepackage{esint}
\usepackage{amsmath}
\usepackage{amssymb,latexsym}
\usepackage{mathrsfs}
\usepackage{graphics}
\usepackage{latexsym}
\usepackage{psfrag}
\usepackage{import}
\usepackage{verbatim}
\usepackage{graphicx}
\usepackage[usenames]{color}
\usepackage{pifont,marvosym}

\theoremstyle{plain}
\newtheorem{lemma}{Lemma}[section]
\newtheorem{theorem}[lemma]{Theorem}
\newtheorem{proposition}[lemma]{Proposition}
\newtheorem{corollary}[lemma]{Corollary}

\theoremstyle{definition}

\newtheorem{definition}[lemma]{Definition}
\newtheorem{remark}[lemma]{Remark}

\numberwithin{equation}{section}

\newcommand{\dom}{\textrm{Dom\,}}

\newcommand{\R}{\mathbb{R}}
\newcommand{\N}{\mathbb{N}}
\newcommand{\Q}{\mathbb{Q}}

\newcommand{\supp}{\text{\rm supp}}

\newcommand{\diam}{{\rm diam\,}}

\newcommand{\ve}{\varepsilon}

\newcommand{\erre}{\mathbb{R}}

\newcommand{\f}{\varphi}

\renewcommand{\P}{\mathcal{P}}

\newcommand{\E}{\mathcal{E}}
\newcommand{\F}{\mathcal{F}}
\renewcommand{\P}{\mathcal{P}}

\newcommand{\RCD}{\mathsf{RCD}}

\newcommand{\CD}{\mathsf{CD}}

\newcommand{\Geo}{{\rm Geo}}
\newcommand{\MCP}{\mathsf{MCP}}
\newcommand{\mm}{\mathfrak m}

\newcommand{\sfd}{\mathsf d}

\newcommand{\Opt}{\mathrm{OptGeo}}
\newcommand{\ee}{{\rm e}}

\begin{document}

\title[Optimal maps in essentially non-branching spaces]{Optimal maps in essentially non-branching spaces}

\author{Fabio Cavalletti}\thanks{F. Cavalletti:  Universit\`a degli Studi di Pavia, Dipartimento di Matematica, email: fabio.cavalletti@unipv.it} 
\author{ Andrea Mondino} \thanks{A. Mondino: The University of Warwick,  Department of Mathematics.  email: A.Mondino@warwick.ac.uk}
%

\bibliographystyle{plain}

\begin{abstract}
In this note we prove that in a metric measure space $(X,\sfd,\mm)$ verifying the measure contraction property with parameters $K \in \R$ and 
$1< N< \infty$, any optimal transference plan between two marginal measures is induced by an optimal map,  provided the first marginal is absolutely continuous with respect to $\mm$ and the space itself is essentially non-branching. In particular this shows that there exists a unique transport plan and it is induced by a map. 
\end{abstract}

\maketitle

\section{Introduction}

One of the first questions of Optimal Transportation theory goes as follows: given two probability measures over a common space 
and a cost function, what is the optimal manner, with respect to this cost, to transport one measure into the other measure? 
This question can be made precise, for instance, by taking as a common space a complete and separable metric space $(X,\sfd)$ and $\sfd^{2}$ as cost function; 
then the optimal transport problem becomes: denoting $\mathcal{P}(X)$ the space of Borel probability measures over $X$ and given $\mu_{0}, \mu_{1} \in \mathcal{P}(X)$, called marginal measures, 
study 
\begin{equation}\label{eq:minW2}
\min_{\pi \in \Pi(\mu_{0},\mu_{1})} \int_{X \times X} \sfd^{2}(x,y)\, \pi(dxdy), 
\end{equation}
where the set of optimal transport plans is defined as follows
$$
\Pi(\mu_{0},\mu_{1}) : = \big\{ \pi \in \mathcal{P}(X\times X) \colon (P_{1})_{\sharp} \pi = \mu_{0}, (P_{2})_{\sharp} \pi = \mu_{1} \big\},
$$
and $P_{i} : X\times X \to X$ denotes the projection on the $i$-th component, for $i = 1,2$. The natural question then is whether the optimal transport plan is induced 
by a transport map or not, i.e. if there exists
$$
T : \dom(T) \subset X \to X, \qquad T_{\sharp}\mu_{0}= \mu_{1},
$$
such that $(Id,T)_{\sharp}\mu_{0} \in \Pi(\mu_{0},\mu_{1})$ is an optimal transport plan, i.e. it is a minimiser in \eqref{eq:minW2}. 

One can easily find examples where such an optimal map cannot exist: if $\mu_{0} = \delta_{o}$ for some $o \in X$ and $\mu_{1}$ is not a Dirac mass, 
then no optimal transport map exists. To avoid such situation, a typical trick is to introduce a reference Radon measure $\mm$ over $X$ such that the metric measure space 
$(X,\sfd,\mm)$ enjoys some ``regularity'' and assume $\mu_{0} \ll \mm$. 
Then metric measure spaces look like a natural framework for proving existence and uniqueness of optimal transport maps.

The problem has a long bibliography.
Existence and uniqueness of optimal transport maps was first proved in the Euclidean setting by Brenier \cite{Brenier} 
under the assumption that the first marginal is absolutely continuous with respect to the Lebesgue measure 
and later extended to more general marginal measures by McCann \cite{McCann1}.
Since then there have been many generalisations; 
most relevant in the context of this paper is the result of McCann \cite{McCann2} for Riemannian
manifolds. 

In the framework of sub-Riemannian manifolds, existence and uniqueness of optimal transport maps  have been established 
by Ambrosio-Rigot \cite{AR} on the Heisenberg group and by Figalli-Rifford \cite{FiRi} under the assumption that the distance 
is locally Lipschitz (or locally semi-concave) outside of the diagonal. For general sub-Riemannian manifolds it seems to be still an open problem. 

For more general metric measure spaces, existence and uniqueness of an optimal transport  map has been 
obtained imposing some type of curvature bounds from below and/or a nice behaviour of the geodesic of the space.
In particular we mention the results by 
\begin{itemize}
\item[-] Bertrand \cite{Bertrand} for Alexandrov spaces;
\item[-] Gigli \cite{GigliGAFA} under the assumption that $(X,\sfd,\mm)$ is a \emph{non-branching} metric measure space satisfying $\CD(K,N)$ (or $\CD(K,\infty)$ under the extra assumption that $\mu_{0}$ is in the domain of the Shannon Entropy); 
\item[-] by the first  author and M. Huesmann \cite{cavahues:existence} assuming $(X,\sfd,\mm)$ to be non-branching 
and $\mm$ to verify a weak property concerning the behaviour of $\mm$ under the shrinking of sets to points. 
This in particular covers non-branching spaces satisfying $\MCP$. Also the cost function could be of the form $h \circ \sfd$ for any increasing and strictly convex $h$;
\item[-]  Rajala and Sturm  \cite{RS2014} under the assumption that $(X,\sfd,\mm)$ satisfies the strong $\CD(K,\infty)$-condition and both $\mu_{0}$ and $\mu_{1}$ are absolutely continuous with respect to $\mm$;
\item[-] Gigli, Rajala and Sturm \cite{GRS2013} for $\RCD(K,N)$ spaces and $\mu_{0}$ absolutely continuous.
\end{itemize}

\noindent
Without the non-branching assumption we cannot expect to have existence and uniqueness of optimal maps, even assuming a lower curvature bound.  In particular in \cite{KR} the authors
construct a branching space satisfying $\MCP(0,3)$ and absolutely continuous marginal measures $\mu_{0}$ and $\mu_{1}$ such that the optimal plan is 
not induced by any map. 
It is then natural to investigate whether a weaker variant of the non-branching condition, namely \emph{essentially non-branching} (see Definition \ref{def:essNB}), 
is enough to obtain existence and uniqueness of optimal maps under curvature bounds. 
The goal of the present paper is to answer affirmatively to such a question. 
Before stating the main result let 
us recall that  an optimal dynamical plan  $\nu\in \Opt(\mu_{0},\mu_{1})$ is given by the map $G:X\to \Geo(X)$ if $\nu=G_{\sharp} \mu_{0}$, 
for all the notations see Section \ref{S:prel}.

\begin{theorem}\label{T:final}
Let $(X,\sfd,\mm)$ be an essentially non-branching metric measure space verifying $\MCP(K,N)$. If $\mu_{0}, \mu_{1} \in \P_{2}(X)$ with $\mu_{0}=\rho_{0} \mm \ll \mm$,  then there exists a unique $\nu \in \Opt(\mu_{0},\mu_{1})$; such a unique   $\nu \in \Opt(\mu_{0},\mu_{1})$ is given by a map and  it satisfies  $(\ee_{t})_{\sharp} \nu=\rho_{t} \mm \ll \mm$ for any $t \in [0,1)$ and  the  $\MCP(K,N)$-inequality 
\begin{equation}\label{E:MCPFinal}
\int \rho_{t}^{1-1/N} \, \mm \geq \int \tau_{K,N}^{(1-t)} (\sfd(x,\ee_{1}(S(x)))) \rho_{0}^{1-1/N} \mm(dx),  \quad  \forall t \in [0,1),
\end{equation}
where $S$ is the unique map giving $\nu$.
 \\ Moreover if $\mu_{0}$ and $\mu_{1}$ have bounded support and $\rho_{0}$ is $\mm$-essentially bounded then 
\begin{align}\label{eq:rhoinfty}
\| \rho_{t}\|_{L^{\infty}(X,\mm)} &\leq \frac{1}{(1-t)^{N}}e^{Dt\sqrt{(N-1)K^{-}}} \| \rho_{0} \|_{L^{\infty}(X,\mm)},  \quad 	\forall t \in [0,1), 
\end{align}
where $D = \diam(\supp (\mu_{0}) \cup \supp (\mu_{1}))$ and $K^{-} = \max\{-K,0\}$.
\end{theorem}

Notice that in particular Theorem \ref{T:final} implies existence and uniqueness of the optimal $W_{2}$-transport map from $\mu_{0}=\rho_{0}\mm \ll \mm$ to $\mu_{1}$. 

Note that in the $\MCP(K,N)$ assumption one requires only a control on geodesics shrinking to  a Delta mass, nevertheless both the estimates \eqref{E:MCPFinal} and \eqref{eq:rhoinfty} are valid for any second marginal $\mu_{1}\in \P_{2}(X)$.
Let us also stress that $\MCP(K,N)$ is the weakest among the finite dimensional Ricci curvature lower bounds conditions; in particular is strictly weaker 
than $\CD$ condition. 

From the technical point of view, notice that while $\MCP$ is a condition on the behaviour of Wasserstein geodesics whose second marginal is a Dirac mass, 
the essentially non-branching property only applies to dynamical optimal geodesics connecting absolutely continuous measures.

\section{Preliminaries}\label{S:prel}
Throughout this note $(X,\sfd)$ is a proper, complete and separable metric space, and $\mm$ is a locally finite non-negative Borel measure (i.e. $\mm(B)< \infty$ for every bounded Borel set $B$).

We will now recall some of the basic objects we will use during the paper. 
For a complete overview on Optimal Transportation theory, we refer to \cite{Vil} and references therein. \\
We will denote with $\Geo(X)$ the set of geodesics of the space, i.e.
$$
\Geo(X) = \{ \gamma \in C([0,1]; X) \colon \sfd(\gamma_{t},\gamma_{s}) = |t-s| \sfd(\gamma_{0},\gamma_{1})\};
$$
for any $t\in [0,1]$ we can consider the evaluation map $\ee_{t} : C([0,1];X) \to X$ defined by $\ee_{t}(\gamma) = \gamma_{t}$.
The set of Borel probability measures over $X$ will be denoted by $\P(X)$, the ones also having  finite second moment are denoted by $\P_{2}(X)$ 
and finally $\P_{ac}(X)$ stands for the set of probability measures absolutely continuous with respect to $\mm$.

We will also consider the set of optimal transference plan  
$$
\Pi_{opt}(\mu_{0},\mu_{1}) : = \left\{ \pi \in \Pi(\mu_{0},\mu_{1}) \colon W_{2}(\mu_{0},\mu_{1})^{2} = \int \sfd^{2}(x,y) \,\pi(dxdy)  \right\};
$$
and the set of optimal dynamical optimal plan 
$$
\Opt(\mu_{0},\mu_{1}) : = \left\{ \nu \in \P(\Geo(X)) \colon \ee_{i\,\sharp}\nu = \mu_{i}, \ i=0,1,\ t \mapsto (\ee_{t})_{\sharp}\nu \ \text{is a } W_{2}\text{-geodesic}  \right\}; 
$$
We  say that an optimal dynamical plan  $\nu\in \Opt(\mu_{0},\mu_{1})$ is given by the map $G:X\to \Geo(X)$ if $\nu=G_{\sharp} \mu_{0}$. 
Notice that in this case in particular the optimal transference plan $(\ee_{0}, \ee_{1})_{\sharp} \nu$ is induced  by the optimal map $\ee_{1}\circ G$.

\begin{definition}
We call a set $\Gamma \subset \Geo(X)$ \emph{non-branching} if for any $\gamma^{1},\gamma^{2} \in \Gamma$ we have: 
if there exists $t \in (0,1)$ such that $\gamma^{1}_{s} = \gamma^{2}_{s}$ for all $s \in [0,t]$, then $\gamma^{1} = \gamma^{2}$.

\medskip
\noindent
A measure $\nu \in \mathcal{P}(\Geo(X))$ is \emph{concentrated on a set of non-branching geodesics} if there exists a non-branching Borel set $\Gamma \subset \Geo(X)$, 
such that $\nu(\Gamma) = 1$.
\end{definition}

Then we recall the following definition given for the first time in \cite{RS2014}.

\begin{definition}\label{def:essNB}
A metric measure space $(X,\sfd,\mm)$ is \emph{essentially non-branching} if for every $\mu_{0},\mu_{1} \in \mathcal{P}_{ac}(X)$,
any $\nu \in \Opt(\mu_{0},\mu_{1})$ is concentrated on a set of non-branching geodesics.
\end{definition}

In order to consider restriction of dynamical plans, for any $s,t \in [0,1]$ with $s \leq t$ we consider the restriction map
$$
\text{restr}^{t}_{s} : C([0,1]; X) \to C([0,1]; X), \qquad \gamma \mapsto \gamma \circ f^{t}_{s},
$$
where $f^{t}_{s} : [0,1] \to [0,1]$ is defined by $f^{t}_{s}(x) = s + (t-s)x$. During this note we will use several times the following fact: 
if $\nu\in \Opt(\mu_{0}, \mu_{1})$ then  the restriction  
$(\text{restr}^{t}_{s})_{\sharp}\nu$ is still an optimal dynamical plan; in particular, called $\mu_{t}:=(\ee_{t})_{\sharp}\nu$, it belongs to $\Opt(\mu_{s}, \mu_{t})$.
This fact simply follows from the triangular inequality of the Wasserstein distance $W_{2}$.

We conclude recalling that for any two measures $\mu_{0},\mu_{1} \in \P(X)$ with $W_{2}(\mu_{0},\mu_{1}) < \infty$, for each $\lambda \in (0,1)$, the set
$$
\mathcal{I}_{\lambda}(\mu_{0},\mu_{1}) := \{ \eta \in \P(X) \colon W_{2}(\mu_{0},\eta) = \lambda W_{2}(\mu_{0},\mu_{1}), \ 
W_{2}(\eta,\mu_{1}) = (1-\lambda) W_{2}(\mu_{0},\mu_{1}) \}, 
$$
is called \emph{set of $\lambda$-intermediate points} and the Excess mass functional, firstly introduced in \cite{R2012}, is defined as follows:
for any $C\geq 0$, $\mathcal{F}_{C} : \P(X) \to [0,1]$ with`
\begin{equation}\label{E:Excessdef}
\mathcal{F}_{C}(\mu) : = \| (\rho - C)^{+} \|_{L^{1}(X,\mm)} + \mu^{s}(X), 
\end{equation}
where $\mu = \rho \mm + \mu^{s}$ with $\mu^{s} \perp \mm$ and $a^{+} = \max\{a,0\}$.

\bigskip


\subsection{Curvature conditions} 
Here we briefly recall the synthetic notions of lower Ricci curvature bounds, for more detail we refer to  \cite{BS10,lottvillani:metric,sturm:I, sturm:II, Vil}.

In order to formulate curvature properties for $(X,\sfd,\mm)$ we introduce the following distortion coefficients: given two numbers $K,N\in \erre$ with $N\geq1$, we set for $(t,\theta) \in[0,1] \times \erre_{+}$,
\begin{equation}\label{E:tau}
\tau_{K,N}^{(t)}(\theta):= 
\begin{cases}
\infty, & \textrm{if}\ K\theta^{2} \geq (N-1)\pi^{2}, \crcr
\displaystyle  t^{\frac{1}{N}}\Bigg(\frac{\sin(t\theta\sqrt{K/(N-1)})}{\sin(\theta\sqrt{K/(N-1)})}\Bigg)^{\frac{N-1}{N}} & \textrm{if}\ 0< K\theta^{2} \leq (N-1)\pi^{2}, \crcr
t & \textrm{if}\ K \theta^{2}<0\ \textrm{or}\\& \textrm{if}\ K \theta^{2}=0\ \textrm{and}\ N=1,  \crcr
\displaystyle  t^{\frac{1}{N}}\Bigg(\frac{\sinh(t\theta\sqrt{-K/(N-1)})}{\sinh(\theta\sqrt{-K/(N-1)})}\Bigg)^{\frac{N-1}{N}} & \textrm{if}\ K\theta^{2} \leq 0 \ \textrm{and}\ N>1.
\end{cases}
\end{equation}
That is, $\tau_{K,N}^{(t)}(\theta): = t^{1/N} \sigma_{K,N-1}^{(t)}(\theta)^{(N-1)/N}$ where
$$
\sigma_{K,N}^{(t)}(\theta) = \frac{\sin(t\theta\sqrt{K/N})}{\sin(\theta\sqrt{K/N})}, 
$$
if $0 < K\theta^{2}<N\pi^{2}$ and with appropriate interpretation otherwise. 
\medskip

\begin{remark}\label{R:estimate}
During the paper we will use the following easy estimate involving the $\sigma$ coefficient for negative $K$: 
\begin{equation}\label{E:sigma}
\frac{\sigma_{K,N}^{(\lambda)}(\theta) }{\lambda} \geq \exp\Big\{ -(1-\lambda) \theta \sqrt{K^{-}/N} \Big\};
\end{equation}
indeed denoting $\alpha = \theta \sqrt{K^{-}/N}$:
$$
\frac{\sigma_{K,N}^{(\lambda)}(\theta) }{\lambda} =   \frac{e^{\lambda \alpha} -e^{-\lambda \alpha}   }{\lambda(e^{ \alpha} -e^{- \alpha})}  
=  e^{- (1-\lambda)\alpha}  \frac{1 -e^{-2\lambda \alpha}   }{\lambda(1 -e^{- 2\alpha})};
$$
Then to prove \eqref{E:sigma} it is sufficient to show $1 -e^{-2\lambda \alpha} \geq \lambda(1 -e^{- 2\alpha})$; that can be rearranged as 
$$
1 - \lambda \geq e^{-2\lambda \alpha} - \lambda e^{-2 \alpha};
$$
but now this last inequality holds for any $\lambda \in [0,1]$: just observe that the right hand side is convex.
\end{remark}

\noindent
Then we also recall the definition of the R\'enyi Entropy functional: $\E_{N} : \P(X) \to [0, \infty]$, 
\begin{equation}\label{E:Entropydef}
\E_{N}(\mu)  = \int_{X} \rho^{1-1/N}(x) \,\mm(dx),
\end{equation}
where $\mu = \rho \mm + \mu^{s}$ with $\mu^{s}\perp \mm$;

\begin{definition}\label{D:CD*}
Let $K \in \R$ and $N \in [1,\infty)$; a metric measure space  $(X,\sfd,\mm)$ verifies $\CD^{*}(K,N)$ if for any two $\mu_{0},\mu_{1} \in \P_{ac}(X)$ 
with bounded support and contained in 
$\supp(\mm)$ there exists $\nu \in \Opt(\mu_{0},\mu_{1})$ such that for any $N'\geq N$
\begin{equation}\label{E:CD*}
\E_{N'}(\mu_{t}) \geq \int \sigma_{K,N'}^{(1-t)} (\sfd(\gamma_{0},\gamma_{1})) \rho_{0}^{-1/N'} 
+ \sigma_{K,N'}^{(t)} (\sfd(\gamma_{0},\gamma_{1})) \rho_{1}^{-1/N'} \,\nu(d\gamma), 
\end{equation}
for any $t\in [0,1]$, where we have written $(\ee_{t})_{\sharp}\nu = \rho_{t} \mm + \mu_{t}^{sing}$ with $\mm \perp \mu_{t}^{sing}$.
\end{definition}

\begin{definition}\label{D:MCP}
Let $K \in \R$ and $N \in [1,\infty)$; a metric measure space  $(X,\sfd,\mm)$ verifies $\MCP(K,N)$ if for any $\mu_{0} \in \P_{ac}(X)$ with bounded  
support and contained in $\supp(\mm)$ and any $o \in \supp(\mm)$ there exists $\nu \in \Opt(\mu_{0},\delta_{o})$ such that 
\begin{equation}\label{E:MCP}
\E_{N}(\mu_{t})  \geq \int \tau_{K,N}^{(1-t)} (\sfd(x,o)) \rho_{0}^{1-1/N} \mm(dx), 
\end{equation}
for any $t\in [0,1)$, where we have written $(\ee_{t})_{\sharp}\nu = \rho_{t} \mm + \mu_{t}^{sing}$ with $\mm \perp \mu_{t}^{sing}$.
\end{definition}

So during this note we will always assume the proper metric measure space $(X,\sfd,\mm)$ to satisfy $\MCP(K,N)$, for some $K,N \in \R$, and to be
essentially non-branching. This will imply that  $\supp(\mm) = X$ and that $(X,\sfd)$ is geodesic.

\section{Good Geodesics under $\MCP$}

Inspired by (and partly following)  the clever work of Rajala  \cite{R2012},\cite{R2013} on the existence of good geodesics in $\CD/\CD^{*}$-spaces, in this section 
we prove the next result roughly stating that in an $\MCP(K,N)$-space we can construct 
Wasserstein geodesics which are absolutely continuous and whose densities satisfy $L^{\infty}$ and Entropy bounds. 
Geodesics satisfying all of these properties will be during this note sometimes named ``good geodesics''.

\begin{theorem}\label{T:goodgeo} 
Let $(X,\sfd,\mm)$ verify $\MCP(K,N)$, for some $K\in \R$ and $N\in [1,\infty)$. Then for any $\mu_{0} \in \P_{ac}(X)$ with bounded 
support and essentially bounded density  and any $o \in X$ there exists $\nu \in \Opt(\mu_{0},\delta_{o})$ such that $(\ee_{t})_{\sharp} \nu \ll \mm$ for every $t \in [0,1)$. Moreover, writing $\mu_{t}:=(\ee_{t})_{\sharp} \nu = \rho_{t} \mm$ for $t\in [0,1)$, 
we have the following upper bound for the density
\begin{equation}\label{E:upperbound}
\| \rho_{t}\|_{L^{\infty}(X,\mm)} \leq \frac{1}{(1-t)^{N}}e^{Dt\sqrt{(N-1)K^{-}}} \| \rho_{0} \|_{L^{\infty}(X,\mm)},  \quad 	\forall t \in [0,1),
\end{equation}
and the following entropy inequality
\begin{equation}\label{E:weakMCP}
\E_{N}(\mu_{t}) \geq (1-t) e^{ - D t  \frac{\sqrt{ (N-1)K^{-}}}{N} } \E_{N}(\mu_{0}),  \quad 	\forall t \in [0,1),
\end{equation}
where $D = \diam(\supp (\mu_{0}) \cup \{ o\})$ and $K^{-} = \max\{-K,0\}$.
\end{theorem}

\begin{proof}
The proof will consist of several steps. We consider $\mu_{0}=\rho_{0} \mm$ and  $\mu_{1} = \delta_{o}$ fixed once for all. Notice moreover that it is sufficient to prove the claim only for $K < 0$. \medskip

{\bf Step 1.} Consider $\lambda \in (0,1)$ fixed.
From $\MCP(K,N)$ there exists $\nu \in \Opt(\mu_{0},\mu_{1})$ verifying \eqref{E:MCP} that is concentrated on a family of geodesics of length at most $D$.
Then from Jensen inequality it follows that 
\begin{align*}
\mm (\{ \rho_{\lambda} > 0 \})^{1/N} \geq  \frac{1}{\|\rho_{0} \|^{1/N}_{L^{\infty}(X,\mm)}} \int_{X} \tau_{K,N}^{(1-\lambda)} (\sfd(x,o))\,\mu_{0}(dx),
\end{align*}
where, as usual, we have written $(\ee_{\lambda})_{\sharp}\nu = \rho_{\lambda} \mm + \mu_{\lambda}^{sing}$ with $\mm \perp \mu_{\lambda}^{sing}$.
Then from \eqref{E:sigma} it follows that
\begin{equation}\label{eq:tauexp}
\tau_{K,N}^{(1-\lambda)} (\theta) = (1-\lambda) \left( \frac{\sigma^{(1-\lambda)}_{K,N-1} (\theta)}{ 1-\lambda} \right)^{\frac{N-1}{N}} \geq (1-\lambda)\exp\big\{ - \theta \lambda \sqrt{ (N-1)K^{-}}/N  \big\}.
\end{equation}
Then 
\begin{equation}\label{E:prima}
\mm (\{ \rho_{\lambda} > 0 \}) \geq  (1-\lambda)^{N}\left( e^{-D\lambda \sqrt{(N-1)K^{-}}} \| \rho_{0} \|_{L^{\infty}(X,\mm)} \right)^{-1}.
\end{equation}
From now on we use the following notation $C(K,N,D,\lambda) : = (1-\lambda)^{-N} e^{\lambda D\sqrt{(N-1)K^{-}}}$.

\medskip

{\bf Step 2.} 
We will need to minimize the excess mass functional \eqref{E:Excessdef} and maximise the R\'enyi Entropy functional on the set of $\lambda$-intermediate points 
$\mathcal{I}_{\lambda}(\mu_{0},\mu_{1})$.

From $\MCP(K,N)$ and the boundedness of the support of $\mu_{0}$, it follows that for any $\lambda \in (0,1)$ the set $\mathcal{I}_{\lambda}(\mu_{0},\mu_{1})$ 
is compact in $(\P(X),W_{2})$, see \cite[Lemma 3.3]{R2012}. 
Moreover for any $C>0$  the excess mass functional  $\mathcal{F}_{C}$ is lower semi-continuous over bounded metric spaces with respect to the 
Wasserstein distance $W_{2}$; therefore,  for any $C\geq 0$ and any $\lambda \in (0,1)$, there exists a minimiser 
of $\mathcal{F}_{C}$ in $\mathcal{I}_{\lambda}(\mu_{0},\mu_{1})$, see  \cite[Lemma 3.6 and Proposition 3.7] {R2012} for additional details.
Moreover $\E_{N}$ is upper semi-continuous and restricted to measures supported on a given bounded set attains only values on a compact interval; 
it follows that  $\E_{N}$ has maximum in $\mathcal{I}_{\lambda}(\mu_{0},\mu_{1})$, see \cite[Lemma 2.4]{R2013} for more details.
\medskip

{\bf Step 3.} Estimate on minimisers of the excess mass functional. \\ 
This part is taken from Proposition 3.11 of \cite{R2012}. We show that for $C> M : = C(K,N,D,\lambda) \| \rho_{0}\|_{L^{\infty(X,\mm)}}$, it holds 
$$
\min_{\eta \in \mathcal{I_{\lambda}}(\mu_{0},\mu_{1})} \mathcal{F}_{C}(\eta) = 0.
$$
We will argue by contradiction. Denote with $\mathcal{I}_{\min} \subset \mathcal{I}_{\lambda}(\mu_{0},\mu_{1})$ the minimisers of $\mathcal{F}_{C}$.
Let $\mu \in \mathcal{I}_{\min}$ be such that 
\begin{equation}\label{eq:MC14}
\mm (\{ \rho_{\mu} > C \}) \geq \left( \frac{M}{C} \right)^{1/4} \, \sup_{\eta \in \mathcal{I}_{\min}} \mm (\{ \rho_{\eta} > C \}),
\end{equation}
where $\mu = \rho_{\mu} \mm + \mu^{s}$ with $\mu^{s} \perp \mm$ and $\eta = \rho_{\eta} \mm + \eta^{s}$ with $\eta^{s} \perp \mm$. 
Consider the set $A : = \{ \rho_{\mu} > C \}$ and assume by contradiction that $\mm(A) > 0$. Then there exists $\delta > 0$ such that 
$$
\mm(A') \geq \left(\frac{M}{C}\right)^{1/2} \mm(A),
$$
with $A' : = \{ \rho_{\mu} > C +\delta \}$. Let $\pi \in \Pi_{opt}(\mu_{0},\mu)$ be an optimal transference plan and  consider any dynamical optimal plan $\tilde \nu$ given by {\bf Step 1} such that 
$$
(\ee_{0})_{\sharp} \tilde \nu =  (P_{1})_{\sharp} \left( \frac{\pi\llcorner_{X \times A'}}{\mu(A')}\right), \qquad (\ee_{1})_{\sharp}  \tilde \nu = \delta_{o}
$$
and verifying \eqref{E:prima}. For ease of notation denote the geodesic $\big((\ee_{s})_{\sharp}\tilde \nu \big)_{s \in [0,1]}$, with $\Gamma$. 
Write then $\Gamma_{\lambda} = \rho_{\Gamma} \mm + \Gamma^{s}$ with $\Gamma^{s} \perp \mm$. 
From \eqref{E:prima} and the definition of $\tilde \nu$ it follows that 
\begin{equation}\label{E:second}
\mm (\{ \rho_{\Gamma} > 0  \}) \geq \frac{\mu(A')}{M} \geq \frac{C}{M} \mm(A') \geq \left(\frac{C}{M}\right)^{1/2} \mm(A).
\end{equation}
Now we can consider a new measure
$$
\tilde \mu = \mu\llcorner_{X \setminus A'} \, + \, \frac{C}{C+\delta} \mu\llcorner_{A'}\, + \, \frac{\delta}{C + \delta} \mu(A') \Gamma_{\lambda}.
$$
From Lemma 3.5 of \cite{R2012} it follows that $\tilde \mu \in \mathcal{I}_{\lambda}(\mu_{0},\mu_{1})$.
Let us compute the variation of the excess mass functional: adopting the usual notation $\tilde \mu = \rho_{\tilde \mu} \mm + \tilde \mu^{s}$ with $\tilde \mu^{s}\perp \mm,$ we have
\begin{align*}
\mathcal{F}_{C}(\mu) - \mathcal{F}_{C}(\tilde \mu)  =&~ \int_{X} (\rho_{\mu} - C)^{+} \, \mm + \mu^{s}(X) -  \int_{X} (\rho_{\tilde \mu} - C)^{+} \, \mm - \tilde \mu^{s}(X)  \crcr
=&~  \int_{X\setminus A'} (\rho_{\mu} - C)^{+} - \left(\rho_{\mu} + \frac{\delta}{C+\delta}\mu(A')\rho_{\Gamma} - C\right)^{+} \,\mm \crcr
&~ +  \int_{ A'} (\rho_{\mu} - C)^{+} - \left(\frac{C}{C+\delta} \rho_{\mu} + \frac{\delta}{C+\delta}\mu(A')\rho_{\Gamma} - C\right)^{+} \,\mm \crcr
&~ + \frac{\delta}{C+\delta} \big( \mu^{s}(A') -\mu(A') \Gamma^{s}(X)  \big) \crcr
=&~  \int_{X\setminus A'} (\rho_{\mu} - C)^{+} - \left(\rho_{\mu} + \frac{\delta}{C+\delta}\mu(A')\rho_{\Gamma} - C\right)^{+} \,\mm \crcr
&~ +  \int_{ A'}  \frac{\delta}{C+\delta}( \rho_{\mu} - \mu(A')\rho_{\Gamma}) \,\mm + \frac{\delta}{C+\delta} \big( \mu^{s}(A') -\mu(A') \Gamma^{s}(X)  \big)  \crcr
=&~  \int_{X\setminus A'} (\rho_{\mu} - C)^{+} - \left(\rho_{\mu} + \frac{\delta}{C+\delta}\mu(A')\rho_{\Gamma} - C\right)^{+} + \frac{\delta}{C+\delta} \mu(A') \rho_{\Gamma} \,\mm \crcr 
=&~  \int_{ \{\rho_{\mu} < C \leq \frac{\delta}{C+\delta} \mu(A') \rho_{\Gamma} +\rho_{\mu}\}   }   (C - \rho_{\mu}) \,\mm  + 
			 \int_{ \{C >\frac{\delta}{C+\delta} \mu(A') \rho_{\Gamma} +\rho_{\mu}\}   }   \frac{\delta}{C+\delta} \mu(A') \rho_{\Gamma} \,\mm \crcr
=&~ \int_{\{\rho_{\mu} < C\}} \min \left\{ C - \rho_{\mu}, \frac{\delta}{C+\delta} \mu(A') \rho_{\Gamma} \right\}\,\mm.
\end{align*}
Since the integrand is non-negative, from the minimality of $\mu$, the integral must be zero and thus $\tilde{\mu}\in  \mathcal{I}_{\min}$. Necessarily 
$$
\mm \left( \{ \rho_{\mu} < C \} \cap \{\rho_{\Gamma} > 0 \} \right) = 0.
$$
Moreover for $y \in \{\rho_{\Gamma} > 0 \} \cap \{\rho_{\mu} \geq C \}$ it holds $\rho_{\tilde \mu} > C$. Hence, from \eqref{E:second} and \eqref{eq:MC14}  we infer
$$
\mm \left( \{\rho_{\tilde \mu} > C\} \right) \geq \mm(\{\rho_{\Gamma} > 0\}) \geq   \left(\frac{C}{M}\right)^{1/2} \mm(A) 
\geq \left( \frac{C}{M} \right)^{1/4} \, \sup_{\eta \in \mathcal{I}_{\min}} \mm (\{ \rho_{\eta} > C \}),
$$
yielding a contradiction, since $\tilde{\mu}\in  \mathcal{I}_{\min}$ and $C>M$.

It remains to  consider the case $\mm(A) = 0$ and $\mu^{s}(X)>0$. This can be treated  analogously by redistributing the  singular the mass using \eqref{E:prima}. This gives a contradiction with the minimality property of $\mu$ since the value of functional ${\mathcal F}_{C}$ evaluated at the  combination of the redistributed singular part and the absolutely continuous part of  $\mu$, is lower than ${\mathcal F}_{C}(\mu)$.

At this point we have  shown that for any $C> C(K,N,D,\lambda) \| \rho_{0}\|_{L^{\infty}(X,\mm)}$, it holds $$
\min_{\eta \in \mathcal{I}_{\lambda}(\mu_{0},\mu_{1})} \mathcal{F}_{C}(\eta) = 0;
$$
with an easy argument one also obtains the same property for $C = C(K,N,D,\lambda)  \| \rho_{0}\|_{L^{\infty}(X,\mm)}$,  see   \cite[Corollary 3.12]{R2012}. 
\\The upper bound just obtained for good intermediate points is the building block to obtain a geodesic from $\mu_{0}$ to $\mu_{1}$ that at each time is absolutely 
continuous with respect to $\mm$ and verifies \eqref{E:upperbound}.

\medskip

{\bf Step 4.} Maximising $\E_{N}$. 
\\To obtain also  the curvature inequality \eqref{E:MCP} one needs to prove the following claim: 
for any maximiser $\mu \in \mathcal{I}_{\lambda}(\mu_{0},\mu_{1})$ of $\E_{N}$ we have $\F_{C}(\mu) = 0$,  where 
$C = C(K,N,D,\lambda) \| \rho_{0}\|_{L^{\infty}(X,\mm)}$.

Such a  claim has been obtained in \cite[Proposition 3]{R2013} under the stronger curvature condition given by the Curvature-Dimension condition $\CD(K,N)$; 
the proof was a modification of \cite[Proposition 3.11]{R2012} that we have already adapted to the weaker curvature condition given by $\MCP(K,N)$ in {\bf Step 3}. 
We therefore take the claim for granted and refer to \cite{R2013} for additional details.

\medskip

{\bf Step 5.} From $\lambda$-Intermediate points to geodesic: upper bound. \\
To summarise: we proved that any maximiser $\mu_{\lambda}$ of $\E_{N}$ restricted to $\mathcal{I}_{\lambda}(\mu_{0},\mu_{1})$ 
(that always exists) is absolutely continuous with respect to $\mm$ and, writing $\mu_{\lambda} = \rho_{\lambda} \mm$,   the estimate \eqref{E:upperbound} holds: 
$$
\| \rho_{\lambda}\|_{L^{\infty}(X,\mm)} \leq \frac{1}{(1-\lambda)^{N}}e^{D\lambda\sqrt{(N-1)K^{-}}} \| \rho_{0} \|_{L^{\infty}(X,\mm)}. 
$$
We denote with $\mathcal{GI}_{\lambda}(\mu_{0},\mu_{1})$ the set of these good $\lambda$-intermediate points.  

To conclude the proof of \eqref{E:upperbound},  we need to prove the same statement but for a complete $W_{2}$-geodesic.  
This part of the proof does not follow  \cite{R2012} and \cite{R2013}; there
the ``globalization'' procedure was built on a bisection argument taking advantage of the symmetric formulation of $\CD$ and $\CD^{*}$ conditions. 
In our framework such a symmetry breaks down and we are forced to proceed with a different argument.
Anyway we recall that a non-symmetric construction was done in Section 5 of \cite{R2012} where it was shown that Definition 2.5 implies $\MCP$ in the sense of Ohta.

Consider $\lambda \in (0,1)$. First we  define recursively a curve in $\P(X)$ only on a countable  subset of $[0,1]$:  
$$
\Gamma_{1} = \mu_{0}, \qquad \Gamma_{(1-\lambda)^{k}} \in \mathcal{GI}_{\lambda}(\Gamma_{(1-\lambda)^{k-1}},\mu_{1}) \quad \forall  k \in \N, k \geq 1.
$$
Then $\Gamma$ is defined on the collection of points $t_{\lambda,k}:= (1-\lambda)^{k},  k \in \N$,  and in particular
$\Gamma_{t_{\lambda, 0}} = \mu_{0}$. Let us prove that on such points both \eqref{E:MCP} and \eqref{E:upperbound} are verified.

We start with the upper bound \eqref{E:upperbound}: if $\Gamma_{t_{\lambda,k}} = \rho^{\lambda}_{k} \mm$ then {\bf Step 4} implies 
$$
\| \rho_{k}^{\lambda}\|_{L^{\infty}(X,\mm)} \leq \frac{1}{(1-\lambda)^{N}}e^{D (1-\lambda)^{k-1} \lambda\sqrt{(N-1)K^{-}}} \| \rho_{k-1}^{\lambda} \|_{L^{\infty}(X,\mm)}, 
$$
where the term $D (1-\lambda)^{k-1}$ comes from the fact that the distance from the support of $\rho_{k-1}^{\lambda}$ and $\delta_{o}=\mu_{1}$ is bounded 
by  $D (1-\lambda)^{k-1}$ since we have taken $k$ subsequent $\lambda$-intermediate points. 
Hence we have
$$
\| \rho_{k}^{\lambda}\|_{L^{\infty}(X,\mm)} \leq \frac{1}{(1-\lambda)^{kN}}e^{D \lambda\sqrt{(N-1)K^{-}} \sum_{1\leq n \leq k}(1-\lambda)^{n-1}} \| \rho_{0} \|_{L^{\infty}(X,\mm)}, 
$$
Calculating the geometric sum yields:

\begin{equation}\label{eq:Linftyboundlambda}
\| \rho_{k}^{\lambda}\|_{L^{\infty}(X,\mm)} \leq \frac{1}{t_{\lambda,k}^{N}}e^{D  (1 - t_{\lambda,k}) \sqrt{(N-1)K^{-}}  }  \| \rho_{0} \|_{L^{\infty}(X,\mm)}.
\end{equation}
Notice that \eqref{eq:Linftyboundlambda} is stable if we let $\lambda\to 0$.
\medskip

{\bf Step 6.} From $\lambda$-Intermediate points to geodesic: Entropy inequality. \\
Since we are considering an optimal transport problem to a Dirac mass, we know that disjoint annular regions centred in the Dirac mass remain disjoint along 
the optimal transport. Then we can restrict ourselves to a sufficiently small annular region 
so that the lengths of the optimal geodesics are almost equal to  $D$ and, 
 by the continuity of the map $\theta \mapsto \tau_{K,N}^{(t)}(\theta)$, assume that all the optimal geodesics have constant length equals to $D$
(for a related reduction argument see the comments after Lemma 3.1 of \cite{R2013}). Using iteratively \eqref{E:MCP} we then have 
$$
\E_{N}(\Gamma_{t_{\lambda,k}}) \geq  \E_{N}(\Gamma_{\lambda,0}) \prod_{n =1}^{k} \tau_{K,N}^{(1-\lambda)}(D(1-\lambda)^{n-1}) \quad \forall k\geq 1;
$$ 
using \eqref{eq:tauexp} we obtain 
\begin{align*}
\E_{N}(\Gamma_{t_{\lambda,k}}) 
\geq &~ \E_{N}(\Gamma_{\lambda,0}) (1-\lambda)^{k}   \exp\left\{ -D \lambda \sum_{n =1}^{k} (1-\lambda)^{n-1} \sqrt{ (N-1)K^{-}}/N    \right\}  \crcr
= &~ \E_{N}(\Gamma_{\lambda,0}) t_{\lambda,k}   \exp\left\{ -D(1-t_{\lambda,k})  \sqrt{ (N-1)K^{-}}/N   \right\}. 
\end{align*}
In particular this implies that 
\begin{equation}\label{eq:entKNlambda}
\E_{N}(\Gamma_{t_{\lambda,k}})  \geq \int  t_{\lambda,k}   \exp\big\{ - \sfd(x,o) (1-t_{\lambda,k})  \sqrt{ (N-1)K^{-}}/N   \big\} \rho_{0}(x)^{-1/N} \,\mu_{0}(dx).
\end{equation}
Notice that \eqref{eq:entKNlambda} is stable for $\lambda\to 0$.

\medskip

In order to conclude,   choose $\lambda:=2^{-j}$ for $j\geq 1$, set $s_{j,k}:=1-(1-2^{-j})^{k}=1-t_{2^{-j},k}$,  for $k \geq 1$,  and define
$$
\bar{\Gamma}^{j}_{s_{j,k}}:=\Gamma^{2^{-j}}_{1-t_{2^{-j},k}} \quad \text{for } k \geq 1, \qquad \text{ and }  
\qquad \bar{\Gamma}^{j}_{0}:=\Gamma^{s^{-j}}_{t_{2^{-j},0}}=\mu_{0}.
$$
Since all the measures $ \{\bar{\Gamma}^{j}_{s_{j,k}}\}_{j,k \in \N}$ are contained in a common compact set, they are precompact in $(\P_{2}, W_{2})$. 
Via a diagonal argument,  letting $j\to \infty$, we then get a limit $W_{2}$-geodesic $\bar{\Gamma}$ 
with $\bar{\Gamma}_{0}=\mu_{0}$ and $\bar{\Gamma}_{1}=\mu_{1}$; 
moreover, since \eqref{eq:Linftyboundlambda} and \eqref{eq:entKNlambda} are  stable if we let  $\lambda \to 0$, 
on the one hand using that uniform density bounds are stable under weak convergence, 
we conclude that the limit geodesic $\bar{\Gamma}_{s} = \rho_{s} \mm$  satisfies the desired bound
$$
\| \rho_{s}\|_{L^{\infty}(X,\mm)} \leq \frac{1}{(1- s)^{N}}e^{Ds \sqrt{(N-1)K^{-}}  }  \| \rho_{0} \|_{L^{\infty}(X,\mm)};
$$
on the other hand using that  the entropy is upper semi-continuous under $W_{2}$-convergence we get that 
\begin{align*}
\int \rho_{s}^{1-1/N} \, \mm \geq &~ (1-s) \int     e^{ - \sfd(x,o) s  \frac{\sqrt{ (N-1)K^{-}}}{N} } \rho_{0}(x)^{1-1/N} \,\mm(dx) \crcr
\geq &~ (1-s) e^{ - D s  \frac{\sqrt{ (N-1)K^{-}}}{N} }  \int     \rho_{0}(x)^{1-1/N} \,\mm(dx).
\end{align*}
\end{proof}

\medskip


\section{Existence of good geodesics for general second marginal}
The goal of this section is to prove the following  result.

\begin{theorem}\label{thm:MCPmu1}
Let $(X,\sfd,\mm)$ be an essentially non-branching m.m.s. satisfying the  $\MCP(K,N)$ condition. Let $\mu_{0}, \mu_{1} \in \P(X)$ have bounded 
support and  assume $\mu_{0}=\rho_{0} \mm \ll \mm$ with $\rho_{0}$ essentially bounded.

 Then  there exists $\nu \in \Opt(\mu_{0},\mu_{1})$  such that   $(\ee_{t})_{\sharp} \nu \ll \mm$ for every $t\in [0,1)$. Moreover, denoting $(\ee_{t})_{\sharp} \nu=\mu_{t}= \rho_{t} \mm$ for any $t \in [0,1)$,  the upper bound \eqref{E:upperbound} and the entropy inequality \eqref{E:weakMCP} hold  with  $D = \diam(\supp (\mu_{0}) \cup \supp (\mu_{1}))$.
\end{theorem}

The rough idea for proving Theorem \ref{thm:MCPmu1} is to approximate the measure $\mu_{1}$ by a convex combination of 
Dirac masses for which we know the validity of  \eqref{E:upperbound} and \eqref{E:weakMCP} thanks to the results of the previous section, and then conclude by a stability argument.  
Such a trick is not new in the literature, see for instance \cite{FigalliJul} where it is used to obtain absolute continuity of Wasserstein geodesics in the Heisenberg group 
and \cite{biacava:streconv} where is used to study $L^{1}$ Optimal transportation problems.
In order to perform such a strategy we will make use of various auxiliary results.  The first one is the following proposition, 
which constitutes a special case of the main Theorem \ref{T:1}.

\begin{proposition}\label{prop:findelta}
Let $(X,\sfd,\mm)$ be an essentially non-branching m.m.s. verifying  $\MCP(K,N)$. 
Let $\mu_{0} \in \mathcal{P}(X)$ with $\mu_{0} \ll \mm$ and with bounded support,  and  $\mu_{1}$ be a finite convex combination of Dirac masses, 
i.e. $\mu_{1}:=\sum_{j=1}^{n} \lambda_{j} \delta_{x_{j}}$ for some $\{x_{j}\}_{j=1,\ldots,n}\subset X$ with $x_{i}\neq x_{j}$ for $i\neq j$, 
and  $\{\lambda_{j}\}_{j=1,\ldots,n}\subset (0,1]$ with $\sum_{j=1}^{n} \lambda_{j}=1$.

Then  there exists a unique transference plan from $\mu_{0}$ to $\mu_{1}$ and it is induced by a map $T$, i.e.
$$
W_{2}(\mu_{0}, \mu_{1} )^{2} =\int_{X} \sfd(x,T(x))^{2} \, \mu_{0}(dx).
$$
\end{proposition}


\begin{proof}
We divide the proof in steps. 
\medskip

{\bf Step 1.}\\
Consider a couple of Kantorovich potentials $\f,\f^{c}$ associated with the transport problem from $\mu_{0}$ to $\mu_{1}$, the sets 
\begin{equation}\label{def:GammaGammax}
\Gamma = \left\{ (x,y) \in X \times X \colon \f(x) + \f^{c}(y) = \frac{\sfd^{2}(x,y)}{2} \right\}, \qquad \Gamma(x) : =P_{2} \Big(\Gamma \cap ( \{x\}   \times X ) \Big), 
\end{equation}
and $S$ the set of those $x \in X$ such that $\Gamma(x)$ is not a singleton. Note that the set $S$ is analytic. 
It will be enough to prove the stronger statement $\mu_{0}(S) = 0$. 

So suppose by contradiction $\mu_{0}(S) > 0$.  Since $\mu_{1}$ is a finite sum of Dirac masses, up to taking a smaller $S$ and up to relabelling the points $x_{j}$,  we can assume that there exist
$$
T_{1},T_{2} : S \to X, \qquad \text{graph}(T_{1}), \ \text{graph}(T_{2}) \subset \Gamma,
$$
both $\mu_{0}$-measurable with $T_{1}(x)=x_{1}$ and $T_{2}(x)=x_{2}$ for all $x \in S$, with $x_{1} \neq x_{2}$ and $S$ is bounded.
\medskip

{\bf Step 2.} \\
With no loss of generality we can assume $\mu_{0}$ to be restricted and renormalised to $S$. 
In particular we redefine $\mu_{0} : = \mm\llcorner_{S}/\mm(S)$. 
Let $\nu^{1} \in \Opt(\mu_{0},\delta_{x_{1}})$ and $\nu^{2} \in \Opt(\mu_{0},\delta_{x_{2}})$ given by Theorem \ref{T:goodgeo}. 
Note  that necessarily 
$$
\nu^{1} \perp \nu^{2};
$$
indeed for $i = 1,2$ it holds $\nu^{i}(\{\gamma \colon \gamma_{1}=x_{i} \}) = 1$ and by construction $x_{1}\neq x_{2}$. 
\\In particular it holds for $i = 1,2$ 
\begin{equation} \label{eq:rhotitau1}
\int (\rho_{t}^{i})^{1-1/N} \, \mm \geq (1-t) e^{ - D t  \frac{\sqrt{ (N-1)K^{-}}}{N} } \int \rho_{0}^{1-1/N} \, \mm(dx) 
=(1-t) e^{ - D t  \frac{\sqrt{ (N-1)K^{-}}}{N} }  \mm(S)^{1/N} ,
\end{equation}
for any $t \in [0,1)$ where we have written $(\ee_{t})_{\sharp} \nu^{i} = \rho_{t}^{i}\mm$ and $D = \diam(S \cup \{ x_{1}, x_{2} \})$.
Then by  Jensen's inequality we get
\begin{eqnarray}
\int_{X} (\rho_{t}^{i})^{1-1/N} \, \mm  &= & \mm(\{\rho_{t}^{i}>0\}) \int_{\{\rho_{t}^{i}>0\}} (\rho_{t}^{i})^{1-1/N} \, \frac{\mm}{\mm(\{\rho_{t}^{i}>0\})}  \nonumber\\
                                                          &\leq&   \mm(\{\rho_{t}^{i}>0\})  \left( \int_{\{\rho_{t}^{i}>0\}} \rho_{t}^{i} \, \frac{\mm}{\mm(\{\rho_{t}^{i}>0\})}  \right)^{1-1/N}    \nonumber\\
                                                          &=&   \mm(\{\rho_{t}^{i}>0\})^{1/N},  \nonumber
\end{eqnarray}
which, combined with \eqref{eq:rhotitau1}, gives

\begin{equation} \label{eq: mmrhot>01}
\liminf_{t\to 0} \mm \left(\{ \rho_{t}^{i} > 0 \} \right) \geq\mm \left(S\right) = \mm \left(\{ \rho_{0}^{i} > 0 \} \right). 
\end{equation}
Denote now by $S^{\varepsilon}:=\{x \in X\,:\, \sfd(x,y)\leq \varepsilon \text{ for some } y \in S\}$ the $\varepsilon$-tubular neighbourhood of $S$ and observe  that,  by Dominated Convergence Theorem, we have $\lim_{\varepsilon\to 0} \mm(S^{\varepsilon})=\mm(S)$. In particular there exists $\varepsilon_{0}>0$ such that 
\begin{equation}\label{eq:3/21}
\mm(S^{\varepsilon_{0}}) \leq \frac{3}{2} \mm(S).
\end{equation}
We now claim that there  exists  a small positive time $\tau > 0$, such that 
\begin{equation}\label{eq:tCharTau11}
\mm \left( \{ \rho_{\tau}^{1} > 0 \} \cap \{ \rho_{\tau}^{2} > 0 \} \right) > 0.
\end{equation}
To this aim notice that, by construction, for  $\mu_{t}^{i}$-a.e. $x \in X$ there exists a geodesic $\gamma \in \Geo (X)$ such that $x=\gamma_{t}$ and $\gamma_{0}\in S$; in particular, for $t \in [0, \varepsilon_{0}]$ the measure $\mu_{t}^{i}$ is concentrated on $S^{\varepsilon_{0}}$.
But then the combination of  \eqref{eq: mmrhot>01} and  \eqref{eq:3/21} implies that there exists $\tau\in (0, \varepsilon_{0})$ satisfying the claim \eqref{eq:tCharTau11}.

\medskip

{\bf Step 3.}\\ 
Note  that  $(\ee_{0},\ee_{1})_{\sharp}(\nu^{1} + \nu^{2})/2$ is an optimal transference plan; indeed 
$$
(\ee_{0},\ee_{1})_{\sharp}(\nu^{1} + \nu^{2})/2 \in \Pi(\mu_{0}, (\delta_{x_{1}}+\delta_{x_{2}})/2)
$$
and since the graph of both $T_{1}$ and $T_{2}$ are subsets of $\Gamma$,  necessarily the transference plan $[(Id,T_{1})_{\sharp} \mu_{0}  +  (Id,T_{2})_{\sharp} \mu_{0} ]/2$ is optimal.
It follows that  
$$
W_{2}(\mu_{0}, (\delta_{x_{1}}+\delta_{x_{2}})/2)^{2} = \frac{1}{2}\sum_{i = 1,2}\int \sfd^{2}(x,y) (Id,T_{i})_{\sharp} \mu_{0} 
= \frac{1}{2}\sum_{i = 1,2}\int \sfd^{2}(x,y) (\ee_{0},\ee_{1})_{\sharp} \nu^{i}. 
$$
The same argument also ensures  that there exists a set $\bar \Gamma \subset \Geo(X)$ 
such that $\{ (\gamma_{0},\gamma_{1}) \colon \gamma \in \bar \Gamma \}$ is $\sfd^{2}$-cyclically monotone (since it is contained in $\Gamma$ which is $\sfd^{2}$-cyclically monotone) 
and $\nu^{1}(\bar \Gamma) = \nu^{2}(\bar \Gamma) = 1$.
In particular, for $t \in (0,1)$ also $(\ee_{0},\ee_{t})_{\sharp}(\nu^{1} + \nu^{2})/2$ is an optimal transference plan and, moreover from Theorem \ref{T:goodgeo}, it follows that 
$(\ee_{t})_{\sharp}(\nu^{1} + \nu^{2})/2$  is absolutely continuous with respect to $\mm$.

We now reach a contradiction obtaining a branching dynamical transference plan between  $(\ee_{0})_{\sharp}(\nu^{1} + \nu^{2})/2$ and $(\ee_{t})_{\sharp}(\nu^{1} + \nu^{2})/2$.
This last part of the proof is strongly inspired by  a clever  mixing procedure performed in  \cite[Corollary 1.4]{RS2014} 
(notice there are some slight differences though).

\medskip

Let $\tau\in (0,1)$ be given by \eqref{eq:tCharTau11} of Step 2 and after using the restriction map, we can also assume that 
$(\ee_{1})_{\sharp}(\nu^{1} + \nu^{2})/2$  is absolutely continuous with respect to $\mm$. Define 
$$
\nu^{\text{left}} = \frac{1}{2} \left(  (\text{restr}^{\tau}_{0})_{\sharp} \nu^{1} + (\text{restr}^{\tau}_{0})_{\sharp} \nu^{2}  \right), 
\qquad 
\nu^{\text{right}} = \frac{1}{2} \left(  (\text{restr}^{1}_{\tau})_{\sharp} \nu^{1} + (\text{restr}^{1}_{\tau})_{\sharp} \nu^{2}  \right). 
$$
Observe that 
$$
(\ee_{1})_{\sharp} \nu^{\text{left}} =  (\ee_{0})_{\sharp} \nu^{\text{right}},
$$
and denote this measure by $\alpha$. We then  consider the associated disintegrations  
$$
\nu^{\text{left}} = \int \nu^{\text{left}}_{x} \, \alpha(dx), \qquad  \nu^{\text{right}} = \int \nu^{\text{right}}_{x} \, \alpha(dx);
$$
in other words $\{\nu^{\text{left}}_{x}\}$ (resp. $\{\nu^{\text{right}}_{x}\}$) is the disintegration of $\nu^{\text{left}}$ (resp. $\nu^{\text{right}}$) with respect to $\ee_{1}$ (resp. $\ee_{0}$). 
The next step is to glue together $\nu^{\text{left}}_{x}$ to $\nu^{\text{right}}_{x}$: consider the map  
$$
Gl : \{ (\gamma^{1},\gamma^{2}) \in C([0,1];X) \times C([0,1];X) \colon \gamma^{1}_{1} = \gamma^{2}_{0} \} \to C([0,1];X), 
$$
defining $Gl(\gamma^{1},\gamma^{2})$ to be equal to $\gamma^{1}_{2s}$ if $0 \leq s \leq 1/2$ and equal to $\gamma^{2}_{2s - 1}$ when $s \geq 1/2$.
We can then set
$$
\nu^{\text{mix}} := \int \nu_{x} \, \alpha(dx), \qquad \nu_{x} : = Gl_{\sharp}(\nu_{x}^{\text{left}}, \nu_{x}^{\text{right}}).
$$
By construction, the measure $\nu^{\text{mix}}$ is concentrated on the set $\tilde{\Gamma}$ defined by
$$
\tilde{\Gamma}:=\{\gamma \in C([0,1]; X) \colon \exists \ \gamma^{1},\gamma^{2} \in \bar \Gamma \colon  \text{restr}_{0}^{\tau} \gamma = \text{restr}_{0}^{\tau} \gamma^{1}, 
\text{restr}_{\tau}^{1} \gamma = \text{restr}_{\tau}^{1} \gamma^{2} \}.
$$
Recalling that  $(\ee_{0}, \ee_{1})(\bar{\Gamma})$ is $\sfd^{2}$-cyclically monotone and using the triangular inequality at  time $\tau$, for $\gamma^{1}, \gamma^{2} \in \bar \Gamma$ with $\gamma^{1}_{\tau} = \gamma^{2}_{\tau}$ we get 
\begin{align*}
\sfd^{2}(\gamma^{1}_{0},\gamma^{1}_{1}) + \sfd^{2}(\gamma^{2}_{0},\gamma^{2}_{1}) 
& \leq  \sfd^{2}(\gamma^{1}_{0},\gamma^{2}_{1}) + \sfd^{2}(\gamma^{2}_{0},\gamma^{1}_{1}) \crcr
& \leq \Big( \tau \ell(\gamma^{1})  + (1-\tau) \ell(\gamma^{2})\Big)^{2} + \Big( \tau \ell(\gamma^{2})  + (1-\tau) \ell(\gamma^{1})\Big)^{2} \crcr
& = \ell(\gamma^{1})^{2}  +  \ell(\gamma^{2})^{2} - 2\tau(1-\tau) \Big( \ell(\gamma^{1})  -  \ell(\gamma^{2})\Big)^{2} \crcr
& \leq \ell(\gamma^{1})^{2}  +  \ell(\gamma^{2})^{2}  \crcr
& =\sfd^{2}(\gamma^{1}_{0},\gamma^{1}_{1}) + \sfd^{2}(\gamma^{2}_{0},\gamma^{2}_{1}),
\end{align*}
where of course  $\ell(\gamma)$ denotes the length of the curve $\gamma$.
It follows that all the previous inequalities are identities;  in particular $\ell(\gamma^{1})= \ell(\gamma^{2})$ and  for $\alpha$-a.e. $x$ there exists $\ell_{x}\geq 0$ such that $\nu_{x}$ is concentrated on geodesics of length $\ell_{x}$.  
We then infer that
\begin{align*}
 \int \sfd^{2}(\gamma_{0},\gamma_{1}) \ \nu^{\text{mix}}(d\gamma)   &=   \int \ell_{x}^{2} \ \ee_{\tau\,\sharp}(\nu^{\text{mix}})(dx) =   \frac{1}{2}\int \ell_{x}^{2} \ \ee_{\tau\,\sharp}(\nu^{1} + \nu^{2})(dx) \crcr
& = \frac{1}{2}\int \sfd^{2}(\gamma_{0},\gamma_{1}) (\nu^{1} + \nu^{2})(d\gamma) .
\end{align*}
Since $\nu^{\text{mix}}$ has the same marginals as $(\nu^{1} + \nu^{2})/2$, and  the latter  is optimal, we conclude that $\nu^{\text{mix}}$  is optimal too.  

We now reach a contradiction with the essentially non  branching assumption by showing that $\nu^{\text{mix}}$ is not concentrated on 
a set of non-branching geodesics. To this aim recall that 
$$
\alpha (\{ \rho_{\tau}^{1} > 0 \} \cap \{ \rho_{\tau}^{2} > 0 \} )> 0,
$$
and that $(\ee_{1})_{\sharp}\nu_{1}=\delta_{x_{1}} \perp \delta_{x_{2}}=(\ee_{2})_{\sharp} \nu_{2}$. Therefore, for $\alpha$-a.e. $x \in \{ \rho_{\tau}^{1} > 0 \} \cap \{ \rho_{\tau}^{2} > 0 \}$  the measure $\nu^{\text{right}}_{x}$ is not a Dirac mass, and thus the dynamical optimal plan $\nu^{\text{mix}}$ is not concentrated on a set of non-branching geodesics.
Then (since in Definition \ref{def:essNB} it is required that $\mu_{1} \ll \mm$) a contradiction is obtained by restricting $\nu^{mix}$ to some interval $[0,1-\ve]$ such that 
$(e_{1-\ve})_{\sharp}\nu_{1},(e_{1-\ve})_{\sharp}\nu_{2} \ll \mm$ and $(e_{1-\ve})_{\sharp}\nu_{1} \perp (e_{1-\ve})_{\sharp}\nu_{2}$. 
The existence of such $\ve >0$ follows from Theorem \ref{T:goodgeo} and the convergence of $\supp((\ee_{1-\ve})_{\sharp}\nu_{i})$ to $x_{i}$ in the Hausdorff distance as $\ve \to 0$.
The claim follows.
\end{proof}

\begin{proposition}\label{prop:MCPfindelta}
Let $(X,\sfd,\mm)$ be an essentially non-branching  m.m.s.  satisfying $\MCP(K,N)$; 
let $\mu_{0}=\rho_{0} \mm$ and $\mu_{1}$ be as in Proposition  \ref{prop:findelta} with $\rho_{0}$ essentially bounded.  

Then there exists $\nu \in \Opt(\mu_{0}, \mu_{1})$ with  $(\ee_{t})_{\sharp}\nu =\mu_{t}= \rho_{t} \mm \ll \mm$ for any $t \in [0,1)$  
satisfying \eqref{E:upperbound} and \eqref{E:weakMCP}.
%
%
\end{proposition}

\begin{proof}
Let $T$ be the optimal map from $\mu_{0}$ to $\mu_{1}$ given by Proposition \ref{prop:findelta}. 
We can define then 
$$
\mu_{0}^{j} : = \mu_{0}\llcorner_{T^{-1}(x_{j})}. 
$$
Then by Theorem \ref{T:goodgeo} we deduce the existence of a dynamical optimal plan $\nu^{j} \in \Opt( \lambda_{j}^{-1}\mu_{0}^{j},\delta_{x_{j}})$ verifying \eqref{E:upperbound} and 
\eqref{E:weakMCP}.
Then we define  $\nu=\sum_{j=1}^{n} \lambda_{j} \nu^{j}$; observe that 
by definition $(\ee_{0})_{\sharp} \nu=\mu_{0}$ and $(\ee_{1})_{\sharp} \nu=\mu_{1}$; moreover $\nu \in \Opt(\mu_{0},\mu_{1})$.  Let us write $\rho^{j}_{t}\mm:= (\ee_{t})_{\sharp} \nu^{j}$ for $j=1, \ldots,n$ and $t \in [0,1)$. It is then clear that $(\ee_{t})_{\sharp} \nu \ll \mm$ for every $t \in [0,1)$ and $\rho_{t} \mm:=(\ee_{t})_{\sharp} \nu^{j}=\left(\sum_{j=1}^{n} \lambda_{j} \rho^{j}_{t}\right) \mm$.

 We first claim that $\rho_{t}$ satisfies   the upper bound  \eqref{E:upperbound} . To this aim, since every $\rho^{j}_{t}$ satisfies the upper bound  \eqref{E:upperbound},  it is enough to observe that   
\begin{equation}\label{eq: rhotjrhoti}
\mm( \{\rho^{i}_{t}>0\} \cap  \{\rho^{j}_{t}>0\})=0 \quad \forall t \in (0,1), \quad \forall i\neq j;
\end{equation}
indeed if for some $\tau\in (0,1)$ it holds $\mm( \{\rho^{i}_{\tau}>0\} \cap  \{\rho^{j}_{\tau}>0\})>0$ for some $i\neq j$, 
then this would contradict Proposition \ref{prop:findelta}: there would exists 
a transport plan from $\mm\llcorner_{\{\rho^{i}_{t}>0\} \cap  \{\rho^{j}_{t}>0\}}$ to $\delta_{x_{i}} + \delta_{x_{j}}$ (both renormalised) not induced by a map.

Let us now show the validity of \eqref{E:weakMCP}. Recalling  \eqref{eq: rhotjrhoti} and since by construction  $\mu_{0}^{i}\perp \mu_{0}^{j}$ for $i \neq j$, then 
\begin{equation}\label{eq:nujnu}
\E_{N} ((\ee_{t})_{\sharp} \nu)= \int \Big(\sum_{j=1}^{n} \rho^{j}_{t}\Big)^{1-1/N} \, \mm  =  \sum_{j=1}^{n}  \int \big( \rho^{j}_{t}\big)^{1-1/N} \, \mm =  \sum_{j=1}^{n} \E_{N} ((\ee_{t})_{\sharp} \nu^{j}).
\end{equation}
Since by construction  each $\nu^{j}$ satisfies \eqref{E:weakMCP},  we conclude that   $\mu_{t}:= (\ee_{t})_{\sharp} \nu$ satisfies \eqref{E:weakMCP} as well.
\end{proof}

\medskip

In the proof of Theorem \ref{thm:MCPmu1} we will make use of the following compactness lemma  (compare with  \cite[Proposition 4.8]{GMSProc}).

\begin{lemma}\label{lem:CompNun}
Let $(X,\sfd)$ be a complete, proper and separable metric space and $\mu^{n}_{0}, \mu^{n}_{1} \subset   \P(X)$ with uniformly bounded supports.  For every $n\in \N$, let $\nu^{n}\in \Opt(\mu^{n}_{0}, \mu^{n}_{1})$. Then there exist a subsequence $n_{k}$ such that the following holds:
\begin{enumerate}
\item There exist $\mu_{0}, \mu_{1}\in \P(X)$ with bounded support such that  $\mu^{n_{k}}_{0}\rightharpoonup \mu_{0}, \mu^{n_{k}}_{1} \rightharpoonup \mu_{1}$ in $W_{2}$;
\item There exist $\nu \in \Opt(\mu_{0}, \mu_{1})$ such that for every $t \in [0,1] \cap \Q$ 
it holds $(\ee_{t})_{\sharp} \nu^{n_{k}} \rightharpoonup (\ee_{t})_{\sharp} \nu$ in $W_{2}$.
\end{enumerate}
\end{lemma}

\begin{proof}
First of all it is clear that the family  $\{\mu^{n}_{t}\}_{t\in [0,1], n\in \N}$ has uniformly compact support, where we have written $\mu^{n}_{t}:=(\ee_{t})_{\sharp}\nu^{n}$. Indeed,  by assumption  $\mu^{n}_{0}, \mu^{n}_{1}$ are concentrated in a common bounded set $B$;  clearly  the set $B_{t}$ of $t$-midpoints
$$B_{t}:=\{\gamma_{t}\,:\, \gamma_{0}, \gamma_{1} \in B\} $$
is also contained in a possibly larger bounded set, since $\sfd(\gamma_{t}, \gamma_{0})= t \sfd(\gamma_{0}, \gamma_{1})\leq \text{diam}(B)$. But, as  $\mu^{n}_{t}(B_{t})=1$, the claim follows by the properness assumption.
\\

Since the supports of $\{\mu^{n}_{t}\}_{t\in [0,1], n\in \N}$ are contained in a common compact subset, by Prokhorov Theorem and a diagonal argument,  there exists a subsequence $n_{k}$ such that   for every $t \in  [0,1]\cap \Q$ there exists $\mu_{t}\in \P(X)$  satisfying:
$$
 \mu^{n_{k}}_{t} \rightharpoonup \mu_{t} \quad \text{weakly as measures and in $W_{2}$, } \; \forall t \in [0,1]\cap \Q. 
 $$
Therefore for every $s,t\in [0,1]\cap \Q$ we have
\begin{equation}\label{eq:st}
W_{2}(\mu_{s},\mu_{t})=\lim_{k\to \infty} W_{2}(\mu^{n_{k}}_{s}, \mu^{n_{k}}_{t} )= \lim_{k\to \infty} \Big( |t-s| \; W_{2}(\mu^{n}_{0}, \mu^{n}_{1}) \Big)= |t-s| \; W_{2}(\mu_{0},\mu_{1}). 
\end{equation}
It follows that the curve $[0,1]\cap \Q \ni t \mapsto \mu_{t}\in X$ is Lipschitz and then it can be uniquely extended to $[0,1]$; in this way, \eqref{eq:st} still holds for every $s,t \in [0,1]$. But this means that the extended curve  is a $W_{2}$-geodesic which can be represented by a probability measure  $\nu \in \Opt(\mu_{0}, \mu_{1})$ satisfying (2) of the thesis. 
\end{proof}

{\bf{Proof of Theorem \ref{thm:MCPmu1}. }}
%
\\Since $(X,\sfd)$ is separable, it is a classical construction to approximate $\mu_{1} \in \P(X)$ by a convex combination of Dirac masses; i.e. there exist a sequence $\{x_{j}\}_{j \in \N}\subset \supp( \mu_{1})$ and a sequence  $\{\{\lambda_{n,j}\}_{j=1}^{n}\}_{j \in \N}\subset (0,1)$ with $\sum_{j \leq n} \lambda_{n,j}=1$ such that
\begin{equation}\label{mu1ntomu1}
 \sum_{j=1}^{n} \lambda_{n,j} \, \delta_{x_{j}}=: \mu_{1}^{n}\rightharpoonup  \mu_{1} \quad \text{in } W_{2}.
\end{equation}
Since by assumption $(X,\sfd,\mm)$ is  essentially non branching and satisfies   $\MCP(K,N)$,   
by Proposition \ref{prop:MCPfindelta} we know that for every $n\in \N$ there exists $\nu^{n}\in \Opt(\mu_{0}, \mu^{n}_{1})$ 
such that  $(\ee_{t})_{\sharp} \nu^{n}=\mu^{n}_{t}=\rho^{n}_{t} \mm$ satisfies
\begin{align}
\| \rho^{n}_{t}\|_{L^{\infty}(X,\mm)} &\leq \frac{1}{(1-t)^{N}}e^{Dt\sqrt{(N-1)K^{-}}} \| \rho_{0} \|_{L^{\infty}(X,\mm)},  \quad 	\forall t \in [0,1), \label{E:UpperBoundN} \\
\E_{N}(\mu^{n}_{t}) &\geq (1-t) e^{ - D t  \frac{\sqrt{ (N-1)K^{-}}}{N} } \E_{N}(\mu_{0}), \quad  \forall t \in [0,1), \label{E:MCPn}
\end{align}
with  $D = \diam(\supp (\mu_{0}) \cup \supp (\mu_{1}))$.

Now, by Lemma \ref{lem:CompNun}, there exists a  limit $\nu \in \Opt(\mu_{0},\mu_{1})$ such that $(\ee_{t})_{\sharp}\nu^{n} \rightharpoonup (\ee_{t})_{\sharp} \nu$ 
weakly as measures, for every $t \in [0,1]\cap \Q$. Therefore, since the R\'enyi entropy is upper semi-continuous with respect to weak convergence (see for instance \cite[Lemma 1.1]{sturm:II}), 
we infer that  \eqref{E:weakMCP} holds for every $t \in [0,1]\cap \Q$.  
Using again the upper semicontinuity of the R\'enyi entropy for the left hand side, and  the continuity in $t$ of the 
right hand side we conclude that  \eqref{E:weakMCP} holds for every $t \in [0,1]$. Analogously, using \eqref{E:UpperBoundN} and that $L^{\infty}$-bounds on the density are stable under weak convergence, we conclude that  $(\ee_{t})_{\sharp} \nu=\rho_{t} \mm \ll \mm$ for every $t \in [0,1)$ and that \eqref{E:upperbound} holds with  $D = \diam(\supp (\mu_{0}) \cup \supp (\mu_{1}))$.



\hfill$\Box$

\section{Main result}

\begin{theorem}\label{T:1}
Let $(X,\sfd,\mm)$ be an essentially non-branching metric measure space verifying  $\MCP(K,N)$.
Then for any $\mu_{0},\mu_{1} \in \mathcal{P}_{2}(X)$ with $\mu_{0} \ll \mm$, 
there exists a unique transference plan and it is induced by a map $T$, i.e.  
$$
W_{2}(\mu_{0}, \mu_{1} )^{2} =\int_{X} \sfd(x,T(x))^{2} \, \mu_{0}(dx).
$$
\end{theorem}

\begin{proof}
The proof is along the same lines of the proof of Proposition \ref{prop:findelta} but with some (non-completely trivial) modifications. 
\medskip

{\bf Step 1.}\\
Consider a couple of Kantorovich potentials $\f,\f^{c}$ associated with the transport problem, the sets 
\begin{equation} \label{eq:defGammaGammaxThm}
\Gamma = \left\{ (x,y) \in X \times X \colon \f(x) + \f^{c}(y) = \frac{\sfd^{2}(x,y)}{2} \right\}, \qquad \Gamma(x) : =P_{2} \Big(\Gamma \cap ( \{x\}   \times X ) \Big), 
\end{equation}
and $S$ the set of those $x \in X$ such that $\Gamma(x)$ is not a singleton. Note that the set $S$ is analytic. 
It will be enough to prove the stronger statement $\mu_{0}(S) = 0$. 

So suppose by contradiction $\mu_{0}(S) > 0$. By Von Neumann Selection Theorem, there exists
$$
T_{1},T_{2} : S \to X, \qquad \text{graph}(T_{1}), \ \text{graph}(T_{2}) \subset \Gamma,
$$
both $\mu_{0}$-measurable and $\sfd(T_{1}(x), T_{2}(x)) > 0$, for all $x \in S$. 
By Lusin Theorem, there exists a compact set $S_{1} \subset S$ such that the maps $T_{1}$ and $T_{2}$ are both continuous when restricted to $S_{1}$ 
and $\mu_{0}(S_{1}) > 0$. In particular 
$$
\inf_{x \in S_{1}} \sfd(T_{1}(x), T_{2}(x)) = \min_{x\in S_{1}} \sfd(T_{1}(x), T_{2}(x)) = 2r >0.
$$
Then one can deduce the existence of a compact set $S_{2} \subset S_{1}$, again with $\mu_{0}(S_{2})>0$ such that
$$
\{ T_{1}(x) \colon x\in S_{2}  \}\subset B_{r}(z_{1}),\qquad    \{ T_{2}(x)  \colon x\in S_{2} \} \subset B_{r}(z_{2}), 
$$
with $\sfd(z_{1},z_{2}) > 2r$, where $B_{r}(z_{i})$ is the open ball centred in $z_{i}$ and radius $r$, for $i = 1,2$. 
\medskip

{\bf Step 2.} \\
With no loss of generality we can assume $\mu_{0}$ to be restricted and renormalised to $S_{2}$. 
In particular we redefine $\mu_{0} : = \mm\llcorner_{S_{2}}/\mm(S_{2})$;  the following measures
 are well defined as well %
$$
\eta_{1} : = (T_{1})_{\sharp} \mu_{0}, \qquad  \eta_{2} : = (T_{2})_ {\sharp} \mu_{0};
$$
in particular $\eta_{1},\eta_{2} $ are Borel probability measures with
$$ 
\eta_{1} \perp \eta_{2};
$$
notice moreover that we can also assume the supports of $\mu_{0}, \eta_{1}$ and $\eta_{2}$ to be bounded.
By Theorem \ref{thm:MCPmu1} we know there exist $\nu^{1} \in \Opt(\mu_{0},\eta_{1})$ and $\nu^{2} \in \Opt(\mu_{0},\eta_{2})$ verifying 
\eqref{E:upperbound} and \eqref{E:weakMCP}; note  that necessarily 
$$
\nu^{1} \perp \nu^{2};
$$
indeed for $i = 1,2$ it holds $\nu^{i}(\{\gamma \colon \gamma_{1} \in B_{r}(z_{i}) \}) = 1$ and by construction $B_{r}(z_{1}) \cap B_{r}(z_{2}) = \emptyset$. 
\\In particular reasoning as in the proof of Proposition \ref{prop:findelta} (in particular see the proof of   \eqref{eq: mmrhot>01}), from \eqref{E:weakMCP} and Jensen's inequality 
it follows that 
\begin{equation} \label{eq: mmrhot>0}
\liminf_{t\to 0} \mm \left(\{ \rho_{t}^{i} > 0 \} \right) \geq\mm \left(S_{2}\right) = \mm \left(\{ \rho_{0}^{i} > 0 \} \right). 
\end{equation}
where $(\ee_{t})_{\sharp} \nu^{i} = \rho_{t}^{i} \, \mm$.

Denote now by $S_{2}^{\varepsilon}:=\{x \in X\,:\, \sfd(x,y)\leq \varepsilon \text{ for some } y \in S_{2}\}$ the $\varepsilon$-tubular neighbourhood of $S_{2}$ and observe  that,  by Dominated Convergence Theorem, we have $\lim_{\varepsilon\to 0} \mm(S_{2}^{\varepsilon})=\mm(S_{2})$. In particular there exists $\varepsilon_{0}>0$ such that 
\begin{equation}\label{eq:3/2}
\mm(S_{2}^{\varepsilon_{0}}) \leq \frac{3}{2} \mm(S_{2}).
\end{equation}
We now claim that there  exists a small positive time $\tau > 0$, such that 
\begin{equation}\label{eq:tCharTau}
\mm \left( \{ \rho_{\tau}^{1} > 0 \} \cap \{ \rho_{\tau}^{2} > 0 \} \right) > 0.
\end{equation}
To this aim notice that, by construction, for  $\mu_{t}^{i}$-a.e. $x \in X$ there exists a geodesic $\gamma \in \Geo (X)$ such that $x=\gamma_{t}$ and $\gamma_{0}\in S_{2}$; in particular, for $t \in [0, \varepsilon_{0}]$ the measure $\mu_{t}^{i}$ is concentrated on $S_{2}^{\varepsilon_{0}}$.
But then the combination of  \eqref{eq: mmrhot>0} and  \eqref{eq:3/2} implies that there exists $\tau\in (0, \varepsilon_{0})$ satisfying the claim \eqref{eq:tCharTau}.

\medskip

{\bf Step 3.}\\ 
Note  that  $(\ee_{0},\ee_{1})_{\sharp}(\nu^{1} + \nu^{2})/2$ is an optimal transference plan. Indeed 
$$
(\ee_{0},\ee_{1})_{\sharp}(\nu^{1} + \nu^{2})/2 \in \Pi(\mu_{0}, (\eta_{1}+\eta_{2})/2)
$$
and since the graph of both $T_{1}$ and $T_{2}$ are subsets of $\Gamma$,  necessarily the transference plan $[(Id,T_{1})_{\sharp} \mu_{0}  +  (Id,T_{2})_{\sharp} \mu_{0} ]/2$ is optimal;
it follows that  
$$
W_{2}(\mu_{0}, (\eta_{1}+\eta_{2})/2)^{2} = \frac{1}{2}\sum_{i = 1,2}\int \sfd^{2}(x,y) (Id,T_{i})_{\sharp} \mu_{0} 
= \frac{1}{2}\sum_{i = 1,2}\int \sfd^{2}(x,y) (\ee_{0},\ee_{1})_{\sharp} \nu^{i}. 
$$
The same argument also ensures  that there exists a set $\bar \Gamma \subset \Geo(X)$ such that $\{ (\gamma_{0},\gamma_{1}) \colon \gamma \in \bar \Gamma \}$ is $\sfd^{2}$-cyclically monotone (since it is contained in $\Gamma$ which is $\sfd^{2}$-cyclically monotone) and 
$\nu^{1}(\bar \Gamma) = \nu^{2}(\bar \Gamma) = 1$.

To conclude now it is enough to run the mixing procedure performed in \cite[Corollary 1.4]{RS2014} and already used in the final step of the proof of Proposition \ref{prop:findelta}. 
This will produce a branching dynamical transference plan between  absolutely continuous measures yielding a contradiction with the essentially non-branching assumption.
\end{proof}

\begin{theorem}\label{T:2}
Let $(X,\sfd,\mm)$ be an essentially non-branching metric measure space verifying  $\MCP(K,N)$.
Then for any $\mu_{0},\mu_{1} \in \mathcal{P}_{2}(X)$ with $\mu_{0} \ll \mm$, there exists a unique $\nu\in \Opt(\mu_{0}, \mu_{1})$ and such $\nu$ is induced by a map.
\end{theorem}

\begin{proof}
As usual, it is sufficient to show that every $\nu\in \Opt(\mu_{0}, \mu_{1})$ is induced by a map; indeed if  there exist $\nu_{1}\neq \nu_{2} \in \Opt(\mu_{0},\mu_{1})$   then also $\frac{1}{2}(\nu_{1}+\nu_{2})$ would be an element of  $\Opt(\mu_{0},\mu_{1})$ but  $\frac{1}{2}(\nu_{1}+\nu_{2})$ cannot be given by a map.

Assume by contradiction there exists $\nu\in \Opt(\mu_{0}, \mu_{1})$ not induced by a map. In particular, given the disintegration of $\nu$ with respect to $\ee_{0}$
$$
\nu = \int_{X} \nu_{x} \, \mu_{0}(dx), 
$$
there exists a compact subset $D\subset \supp (\mu_{0})$ with $\mu_{0}(D)>0$ such that for $\mu_{0}$-a.e. $x \in D$ the probability measure $\nu_{x}$ is not a Dirac mass. 
Via a selection argument, for $\mu_{0}$-a.e. $x \in D$ we can also assume that $\nu_{x}$ is the sum of two Dirac masses.
Then for $\mu_{0}$-a.e. $x \in D$ there exist $t=t(x)\in (0,1)$ such that $(\ee_{t})_{\sharp} \nu_{x}$ is not a Dirac mass over $X$.
Then by continuity there exists an  open interval $I=I(x)\subset (0,1)$ containing $t(x)$ above such that  $(\ee_{s})_{\sharp} \nu_{x}$ is still not a Dirac mass over $X$, for every $s \in I(x)$.
\\It follows that we can find a subset  $\bar{D} \subset D \subset X$ still satisfying $\mu_{0}(\bar{D})>0$ with the following property: there exists $\bar{q}\in \Q\cap (0,1)$ such that $(\ee_{\bar q})_{\sharp} \nu_{x}$ is not a Dirac mass, for every  $x \in \bar{D}$.

 Indeed, since  $D=  \bigcup_{q \in 	\Q \cap (0,1)} D_{q}$ where 
$$D_{q}:=\left\{ x \in D \,: \, (\ee_{q})_{\sharp} \nu_{x} \text{ is not a Dirac mass} \right\} $$
and since $\mu_{0}(D)>0$, there must exist  $\bar{q}\in 	\Q \cap (0,1)$ with $\mu_{0}(D_{\bar{q}})>0$; we then set $\bar{D}:=D_{\bar{q}}$.
Set now
$$
\bar{\nu} = \frac{1}{\mu_{0}(\bar{D})} \int_{\bar{D}} \nu_{x} \, \mu_{0}(dx). 
$$
Note that $\bar{\nu}$ is an optimal dynamical plan; in particular $(\ee_{0}, \ee_{\bar{q}})_{\sharp} \bar{\nu}$ is an optimal plan which is not given by a map. 
This contradicts Theorem \ref{T:1} and thus the proof is complete.
\end{proof}

We can now collect Theorem \ref{thm:MCPmu1}, Theorem \ref{T:1} and Theorem \ref{T:2} in order to prove  Theorem \ref{T:final}.
 \bigskip 
 \\
 
 \textbf{Proof of Theorem \ref{T:final}}.
 \\ Having Theorem \ref{thm:MCPmu1}, Theorem \ref{T:1} and Theorem \ref{T:2} at disposal, the only non trivial statements to show is that if $\mu_{0},\mu_{1}\in \P_{2}(X)$ with  $\mu_{0}=\rho_{0} \mm \ll \mm$ then  $(\ee_{t})_{\sharp} \nu \ll \mm$ for any $t \in [0,1)$, and that  \eqref{E:MCPFinal} holds.\\

\textbf{Step 1}:  $(\ee_{t})_{\sharp} \nu \ll \mm$ for any $t \in [0,1)$.
\\ To this aim first of all observe that since we know that the transport is given by a  $W_{2}$-optimal  map $T:X \to X$, then there exist partitions $\supp(\mu_{0})=\mathring \bigcup_{j\in \N} E^{0}_{j}, \; \supp(\mu_{1})=\mathring \bigcup_{j\in \N} E^{1}_{j}$ such that $\mu_{0}(E^{0}_{j})>0$  for all $j\in \N$,  each $E^{0}_{j}, E^{1}_{j}$ is bounded, $\|\rho_{0}\|_{L^{\infty}(E^{0}_{j},\mm)}<\infty$ and  for every $j\in \N$ there exists $i \in \N$ such that $T(E^{0}_{j})\subset E^{1}_{i}$.
\\Call $\mu_{0}^{j}:=\mu_{0}(E^{0}_{j})^{-1} \mu_{0}\llcorner E^{0}_{j}$, $\Gamma^{j}:=\{ \gamma \in \Geo(X) \,: \,  \gamma_{0} \in E^{0}_{j}\}$ and let $\nu^{j} := \mu_{0}(E^{0}_{j})^{-1} \nu \llcorner \Gamma^{j}$.
Notice that  $\nu^{j} \in \Opt(\mu_{0}^{j}, T_{\sharp} \mu_{0}^{j})$ and, since by construction $\mu_{0}^{j}$ has bounded support and bounded density  and $T_{\sharp} \mu_{0}^{j}$ has bounded support, then  we can apply  Theorem \ref{T:1} and infer that  $(\ee_{t})_{\sharp} \nu^{j}= \rho^{j}_{t} \mm \ll \mm$ for all $t \in [0,1)$ and $j \in \N$. 
\\Therefore it is enough to show that 
\begin{equation}\label{eq:rhoirhoj>0}
\mm(\{\rho^{j}_{t}>0\} \cap  \{\rho^{i}_{t}>0\} )=0  \quad \forall t \in (0,1),\; \forall i\neq j.
\end{equation}
If by contradiction there exists $\tau\in (0,1)$ such that $\mm(\{\rho^{j}_{\tau}>0\} \cap  \{\rho^{i}_{\tau}>0\} )>0$ for some $i \neq j$ then we could  run the mixing procedure performed in \cite[Corollary 1.4]{RS2014} and already used in the final step of the proof of Proposition \ref{prop:findelta}. 
This will produce a branching dynamical transference plan between  absolutely continuous measures yielding a contradiction with the essentially non-branching assumption. More precisely, following verbatim {Step 3} in the proof  of Proposition \ref{prop:findelta}, we get an optimal dynamical plan $\nu^{mix}$ such that for $\mm$-a.e. $x\in (\{\rho^{j}_{\tau}>0\} \cap  \{\rho^{i}_{\tau}>0\}$ the measure $\nu_{x}^{left}$ is not a Dirac mass; therefore  the time-reversed optimal plan $\nu^{mix,-}$ defined by
$$\nu^{mix,-}(\gamma):=\nu(\gamma^{-}), \quad  \text{ where } \gamma^{-}_{t}:=\gamma_{1-t} \quad \forall t \in [0,1],$$
is not concentrated on a set of non-branching geodesics. Since by construction $(\ee_{t})_{\sharp}(\nu^{mix,-})\ll \mm$ for every $t \in  (0,1]$, the optimal dynamical plan $\bar{\nu}:=\textrm{restr}_{\frac{1-\tau}{2}}^{1} \nu^{mix,-}$   contradicts the essential non-branching assumption.
\\

\textbf{Step 2}: the $\MCP(K,N)$ inequality  \eqref{E:MCPFinal} holds in case $\mu_{1}$ is a finite convex combination of Dirac masses.
\\We first consider the case $\mu_{1}=\sum_{j=1}^{n} \lambda_{j} \delta_{x_{j}}$ with $\lambda_{j}\in (0,1]$ for every $j=1,\ldots,n$, $\sum_{j=1}^{n} \lambda_{j}=1$, and $x_{i}\neq x_{j}$ for $i\neq j$. 
Let $\nu \in \Opt(\mu_{0}, \mu_{1})$ and  consider the disintegration $\nu=\sum_{j=1}^{n} \nu_{j}$ with respect to $\ee_{1}$; observe that by definition $(\ee_{1})_{\sharp} \nu_{j}= \lambda_{j} \delta_{x_{j}}$. Since $\frac{1}{\lambda_{j}} \nu_{j} \in \Opt(\frac{1}{\lambda_{j}} (\ee_{0})_{\sharp} \nu_{j} , \delta_{x_{j}})$, and since by the first part of the theorem  the optimal dynamical plan is unique, then $\nu_{j}$ satisfies the  $\MCP(K,N)$ inequality
\begin{equation}\label{E:MCPj}
\int \rho_{j, t}^{1-1/N'} \, \mm \geq \int \tau_{K,N'}^{(1-t)} (\sfd(x,x_{j})) \rho_{j,0}^{1-1/N'} \, \mm (d\gamma), \quad \forall t \in [0,1), 
 \end{equation}
where we have written $(\ee_{t})_{\sharp}\nu_{j}= \rho_{j,t} \mm$.
\\Using  that the transport from $\mu_{0}$ to $\mu_{1}=\sum_{j=1}^{n} \lambda_{j} \delta_{x_{j}}$ is given by a map, since by construction $x_{i}\neq x_{j}$ for $i\neq j$,  it follows  that the measures $(\ee_{0})_{\sharp} (\nu_{j})$ are concentrated on pairwise disjoint sets. In particular it holds
\begin{eqnarray}
\sum_{j=1}^{n} \int \tau_{K,N'}^{(1-t)} (\sfd(x,x_{j})) \rho_{j,0}^{1-1/N'} \, \mm (d\gamma) &=& \sum_{j=1}^{n} \int \tau_{K,N'}^{(1-t)} (\sfd(\gamma_{0},\gamma_{1})) \rho_{j,0}^{-1/N'} \, \nu_{j} (d\gamma) \nonumber \\
&=&   \int \tau_{K,N'}^{(1-t)} (\sfd(\gamma_{0},\gamma_{1})) \, \Big(  \sum_{j=1}^{n}  \rho_{j,0}^{-1/N'} \nu_{j} \Big) (d\gamma) \nonumber \\
&=&   \int \tau_{K,N'}^{(1-t)} (\sfd(\gamma_{0},\gamma_{1}))  \,  \rho_{0}^{-1/N'} \nu  (d\gamma), \label{eq:nujnu}
\end{eqnarray}
where, as usual,  we  write $(\ee_{t})_{\sharp}\nu= \rho_{t} \mm$, for $t \in [0,1)$.  

On the other hand, we also have
\begin{equation}\label{eq: rhotjrhot3}
\mm(\{ \rho_{t,j}>0\} \cap \{ \rho_{t,i}>0\})=0 \, \quad \forall t \in [0,1], \; \forall i\neq j.
\end{equation}
Indeed  otherwise there exists $\tau\in (0,1)$ (notice that $\tau$ cannot be $0$ by the above argument or $1$ since by construction $\mu_{1}$ is finite sum of Dirac masses) such that $\mm(\{ \rho_{\tau,j}>0\} \cap \{ \rho_{\tau,i}>0\})>0$, and we could repeat verbatim  verbatim {Step 3} in the proof  of Proposition \ref{prop:findelta} arriving to contradict the essential non branching assumption.

Having \eqref{eq: rhotjrhot3} at disposal, we can argue as in \eqref{eq:nujnu} and get that
\begin{equation}\label{eq:rhojtN}
\sum_{j=1}^{n} \int \rho_{j, t}^{1-1/N'} \, \mm=   \int \Big( \sum_{j=1}^{n }\rho_{j, t}^{1-1/N'} \Big) \, \mm = \int \rho_{t}^{1-1/N'} \, \mm, \quad \forall t \in [0,1], \; \forall N'\geq N.
\end{equation}
The combination of \eqref{E:MCPj}, \eqref{eq:nujnu} and \eqref{eq:rhojtN} implies that
$$
\int \rho_{t}^{1-1/N'} =  \sum_{j=1}^{n} \int \rho_{j, t}^{1-1/N'} \, \mm \geq  \sum_{j=1}^{n} \int \tau_{K,N'}^{(1-t)} (\sfd(x,x_{j})) \rho_{j,0}^{1-1/N'} \, \mm (d\gamma) =  \int \tau_{K,N'}^{(1-t)} (\sfd(\gamma_{0},\gamma_{1}))  \,  \rho_{0}^{-1/N'} \nu  (d\gamma),
$$
as desired.
\\

\textbf{Step 3}: the $\MCP(K,N)$ inequality  \eqref{E:MCPFinal} holds for a general  $\mu_{1}\in \P_{2}(X)$.
\\Since $(X,\sfd)$ is separable, it is a classical construction to approximate $\mu_{1} \in \P_{2}(X)$ by a convex combination of Dirac masses; i.e. there exist a sequence $\{x_{j}\}_{j \in \N}\subset \supp( \mu_{1})$ and a sequence  $\{\{\lambda_{n,j}\}_{j=1}^{n}\}_{j \in \N}\subset (0,1)$ with $\sum_{j \leq n} \lambda_{n,j}=1$ such that
\begin{equation}\label{mu1ntomu1}
\sum_{j=1}^{n} \lambda_{n,j} \, \delta_{x_{j}}=: \mu_{1}^{n}\rightharpoonup  \mu_{1} \quad \text{in $W_{2}$.}
\end{equation}
By step 2 we know that for every $n\in \N$ there exists $\nu^{n}\in \Opt(\mu_{0}, \mu^{n}_{1})$ such that 
\begin{equation}\label{E:MCPn}
\int (\rho_{t}^{n})^{1-1/N'} \, \mm \geq \int \tau_{K,N'}^{(1-t)} (\sfd(\gamma_{0},\gamma_{1})) \rho_{0}^{-1/N'} \, \nu^{n} (d\gamma), \quad \forall t \in [0,1), \; \forall N'\geq N,
 \end{equation}
where we have written $(\ee_{t})_{\sharp}\nu^{n} = \rho_{t}^{n} \mm$.  
Now, by \cite[Proposition 4.8]{GMSProc}, there exists a  limit $\nu \in \Opt(\mu_{0},\mu_{1})$ such that $(\ee_{t})_{\sharp}\nu^{n} \rightharpoonup (\ee_{t})_{\sharp} \nu$ in $W_{2}$, for every $t \in [0,1]\cap \Q$. Therefore, since the Renyi entropy is upper semicontinuous with respect to $W_{2}$ convergence (see for instance \cite[Lemma 1.1]{sturm:II}), we infer that  \eqref{E:MCPFinal} holds for every $t \in [0,1]\cap \Q$.  Using again the upper semicontinuity of the Renyi entropy for the left hand side, and  the continuity in $t$ of the right hand side we conclude that  \eqref{E:MCPFinal} holds for every $t \in [0,1]$.
\\ \hfill$\Box$

Combining our main result Theorem \ref{T:final} with the work of Rajala \cite{R2012}, \cite{R2013} for $\CD(K,N)/\CD^{*}(K,N)$ spaces, we get the next corollary.

\begin{corollary}\label{Cor:CD}
Let $(X,\sfd,\mm)$ be an essentially non-branching metric measure space verifying  $\CD(K,N)$ (resp. $\CD^{*}(K,N)$) for some $K \in \R, N\in (1,\infty)$. If $\mu_{0}, \mu_{1} \in \P_{2}(X)$ with $\mu_{0}=\rho_{0} \mm \ll \mm$,  then there exists a unique $\nu \in \Opt(\mu_{0},\mu_{1})$; such a unique   $\nu \in \Opt(\mu_{0},\mu_{1})$ is given by a map and  it satisfies  $(\ee_{t})_{\sharp} \nu \ll \mm$ for any $t \in [0,1)$. 
\\ Moreover if $\mu_{0}=\rho_{0}\mm, \mu_{1}=\rho_{1}\mm \in \P_{ac}(X)$  have bounded support,  then $\nu$ satisfies the $\CD(K,N)$-convexity (respectively $\CD^{*}(K,N)$-convexity) condition and if in addition the densities $\rho_{0}, \rho_{1}$ are $\mm$-essentially bounded then, writing $(\ee_{t})_{\sharp} \nu=\rho_{t} \mm$, it holds 
\begin{align}
\| \rho_{t}\|_{L^{\infty}(X,\mm)} &\leq  e^{D\sqrt{(N-1)K^{-}}} \max\{ \| \rho_{0} \|_{L^{\infty}(X,\mm)},   \| \rho_{1} \|_{L^{\infty}(X,\mm)} \}, \quad 	\forall t \in [0,1], \text{  if $\CD(K,N)$ holds} \label{eq:boundCD}, \\
\| \rho_{t}\|_{L^{\infty}(X,\mm)} &\leq  e^{D\sqrt{N \,K^{-}}} \max\{ \| \rho_{0} \|_{L^{\infty}(X,\mm)},   \| \rho_{1} \|_{L^{\infty}(X,\mm)} \}, \quad 	\forall t \in [0,1], \text{  if $\CD^{*}(K,N)$ holds,} \label{eq:boundCD*}
\end{align}
where $D = \diam(\supp (\mu_{0}) \cup \supp (\mu_{1}))$ and $K^{-} = \max\{-K,0\}$.
 \end{corollary}
 
\begin{proof}
 Since $\CD(K,N)$ implies  $\MCP(K,N)$ and $\CD^{*}(K,N)$ imply $\CD(K^{*},N)$ for $K=K \frac{N-1}{N}$ (see \cite[Proposition 2.5]{BS10}), by Theorem \ref{T:final}  there exists a unique $\nu \in \Opt(\mu_{0},\mu_{1})$ which moreover  is given by a map and  satisfies  $(\ee_{t})_{\sharp} \nu \ll \mm$ for any $t \in [0,1)$. By uniqueness of the optimal dynamical plan,  if $\mu_{0}=\rho_{0}\mm, \mu_{1}=\rho_{1}\mm \in \P_{ac}(X)$  have bounded support, then it is obvious that $\nu$ satisfies the $\CD(K,N)$-convexity (respectively $\CD^{*}(K,N)$-convexity)  condition. Finally the $L^{\infty}$-bounds on the density $\rho_{t}$ of $(\ee_{t})_{\sharp} \nu$  follow from \cite[Theorem 1.3]{R2012} for the $\CD(K,N)$ case, and from  \cite[Theorem 1.2]{R2013}  for the  $\CD^{*}(K,N)$ case.
 \end{proof}

We can also obtain existence and uniqueness of optimal maps under local curvature conditions.
\begin{corollary}\label{C:CDloc}
Let $(X,\sfd, \mm)$ be an essentially non-branching, proper, geodesic, metric measure space satisfying $\CD_{loc}(K,N)$. 
Then for any $\mu_{0}, \mu_{1} \in \P_{2}(X)$ with $\mu_{0} \ll \mm$ there exists a unique optimal transport map. 
Moreover there exists a unique optimal dynamical plan $\nu \in \Opt(\mu_{0},\mu_{1})$ and $\mu_{t} =\rho_{t} \mm$ for any $t \in [0,1)$.
\end{corollary}

\begin{proof}
Consider any $\pi \in \Pi(\mu_{0},\mu_{1})$ optimal transference plan  and any $\nu \in \Opt(\mu_{0},\mu_{1})$ associated to it; 
let moreover $\mu_{t} = (\ee_{t})_{\sharp} \nu$.
Observe that it is not restrictive to assume both $\mu_{0}$ and $\mu_{1}$ to have bounded support; in particular 
$D : = \diam(\supp(\mu_{0})\cup \supp(\mu_{1}))$ is finite and therefore there exists a compact set $B \subset X$  
such that $\supp(\mu_{t}) \subset B$ for each $t \in [0,1]$. 

From the $\CD_{loc}$ condition we deduce the existence of an open covering of $X$, denoted by $\{ U_{i} \}_{i \in I}$ where $\CD(K,N)$ holds for
marginal measures supported inside the same $U_{i}$; in particular we deduce from Theorem \ref{T:final} the existence and uniqueness of optimal transport maps 
for marginal measures supported inside the same $U_{i}$, provided the first one is absolutely continuous with respect to $\mm$.

Since $B$ is a compact set, from Lebesgue's number Lemma, there exists $\delta> 0$ such that whenever $A \subset B$ has diameter less than $\delta$ then 
it is contained in $U_{i}$, for some $i \in I$.
Now consider any $t \in [0,1)$ and consider the disintegration of $\nu$ with respect to $\ee_{t}$:
$$
\nu = \int \nu_{x}^{t} \, \mu_{t}(dx)
$$
and observe that observe that 
$$
\diam\left( \supp\left( \mu_{t}\llcorner_{B_{\delta/4}(z)} \right) 
\cup \supp\left(  (\ee_{t + \delta/4D} )_{\sharp} \Big( \int_{B_{\delta/4}(z)} \nu_{x}^{t} \, \mu_{t}(dx) \Big)\right)\right) < \delta.
$$
This implies that there exists a unique optimal map from $\mu_{t}$ to $\mu_{t+\delta/4D}$, provided $\mu_{t}$ is absolutely continuous with respect to $\mm$. 
If this is the case, from Theorem \ref{T:final} it also follows that $\mu_{t+\delta/5D}$ is absolutely continuous with respect to $\mm$.

Since $t$ was any number in $[0,1)$, we can start with $t= 0$; since $\mu_{0} \ll \mm$, there exists a unique optimal map $T_{0}$ such 
that $T_{\sharp} \mu_{0} = \mu_{\delta/5D}$ and $\mu_{\delta/5D} \ll \mm$. Repeating the argument finitely many times, it follows the existence of a map $T$ 
such that $\pi = (Id,T)_{\sharp}\mu_{0}$ and the claim follows.
Repeating verbatim the proof of Theorem \ref{T:2}, we obtain that the optimal dynamical plan $\nu$ is unique and it is induced by a map. 
The absolute continuity follows from the covering argument of the first part of the proof.
\end{proof}

We conclude by saying that  Corollary \ref{C:CDloc} permits to extend the result of \cite{cavasturm:MCP} to the framework of essentially non-branching metric measure spaces.
In particular it follows that an essentially non-branching metric measure spaces verifying $\CD_{loc}(K,N)$ also verifies $\MCP(K,N)$; 
hence the claims of Theorem \ref{T:final} are still valid.

\section*{Appendix}

The Ricci curvature condition $\MCP$ was introduced independently in \cite{Ohta1} and \cite{sturm:II}. 
The definition proposed by S.I. Ohta in \cite{Ohta1} goes as follows:
\medskip

A metric measure space $(X,\sfd,\mm)$ is said to satisfy $\MCP(K,N)$ if for any $x \in X$ and $A \subset X$ Borel set with $0 < \mm(A) <\infty$ 
there exists 
$$
\nu \in \Opt(\mm\llcorner_{A}/\mm(A), \delta_{x} )
$$
such that 
$$
\mm \geq (\ee_{t})_{\sharp} \big( \tau_{K,N}^{(1-t)}(\sfd(\gamma_{0},\gamma_{1}))^{N} \mm(A) \nu \big).
$$
That is, using the estimate \eqref{E:sigma}, for any $B \subset X$ it holds
\begin{align*}
\mm(B) \geq &~ \mm(A) \int_{\ee_{t}^{-1}(B)}\tau_{K,N}^{(1-t)}(\sfd(\gamma_{0},\gamma_{1}))^{N} \nu(d\gamma)  \\
\geq &~ \mm(A) (1-t)^{N}e^{ - D t \sqrt{ (N-1)K^{-}}} (\ee_{t})_{\sharp}\nu (B),
\end{align*}
where $D = \diam (A \cup \{x\})$.
In particular this implies that $(\ee_{t})_{\sharp}\nu = \rho_{t} \mm$, for any $t \in[0,1)$ and for $\nu$-a.e. $\gamma$
$$
\rho_{t}(\gamma_{t})  \leq \frac{1}{\mm(A)} (1-t)^{-N}e^{ D t \sqrt{ (N-1)K^{-}}} = \rho_{0}(\gamma_{0}) (1-t)^{-N}e^{ D t \sqrt{ (N-1)K^{-}}} 
$$
that rearranged properly becomes
\begin{equation}\label{E:MCP3}
\rho_{t}(\gamma_{t})^{-1/N} \geq (1-t)e^{ -D t \sqrt{ (N-1)K^{-}}/N} \rho_{0}(\gamma_{0})^{-1/N},
\end{equation}
yielding all the claims of Theorem \ref{T:goodgeo} under the additional assumption that $\mu_{0} = \mm\llcorner_{A}/\mm(A)$. 
The previous calculations show that the claim of Proposition \ref{prop:findelta} can be proved assuming the space to satisfy $\MCP$ version of Ohta 
and essentially non-branching (see in particular {\bf Step 2.} of the proof). 

This permits to approximate any $\mu_{0}$ with simple functions (i.e. finite linear combination of characteristic functions) and, thanks to Proposition \ref{prop:findelta},
to obtain a Wasserstein geodesic connecting the approximation of $\mu_{0}$ to the finite combination of Dirac masses $\mu_{1}$ satisfying the  estimate \eqref{E:MCP3}. 
Since $\eqref{E:MCP3}$ is stable, we directly obtain also Proposition \ref{prop:MCPfindelta}.
 Then one can repeat completely all the rest of the paper using the Ohta's version of $\MCP$ obtaining the same claims. 

As a consequence we also obtain that if $(X,\sfd,\mm)$ is an essentially non-branching metric measure space it satisfies
Ohta's version of $\MCP$ if and only if it satisfies Definition \ref{D:MCP}. 
Indeed we have shown that under the essentially non-branching condition both $\MCP$ definitions can be considered as pointwise conditions on the density of the 
Wasserstein geodesics connecting absolutely continuous measures to a Dirac mass and as pointwise condition they impose the same inequality: 
for $\nu$-a.e. $\gamma$
$$
\rho_{t}(\gamma_{t})^{-1/N} \geq \tau_{K,N}^{(1-t)}(\sfd(\gamma_{0},\gamma_{1})) \rho_{0}(\gamma_{0})^{-1/N},
$$
for every $t \in [0,1)$.
We conclude this part noticing that, by Section 5 of \cite{R2012}, Definition \ref{D:MCP} implies $\MCP$ in the sense of Ohta even without the essential non-branching assumption.

\end{document}